\documentclass[a4paper,11pt]{article}
\usepackage[english]{babel}
\usepackage[utf8]{inputenc}
\usepackage{lmodern}
\usepackage{amsfonts}
\usepackage{bm}
\usepackage{amssymb}
\usepackage{latexsym}
\usepackage{amsmath}
\usepackage{amscd}
\usepackage{amsthm}
\usepackage{pifont}
\usepackage{fancyhdr}
\usepackage{epsfig}
\usepackage{alltt}
\usepackage{hyperref}
\usepackage{multimedia}
\usepackage[utf8]{inputenc}
\usepackage[T1]{fontenc}
\usepackage{psfrag}
\usepackage{url}
\usepackage{graphicx}
\usepackage{psfrag}

\newtheorem{thm}{Theorem}
\newtheorem{defn}{Definition}
\newtheorem{prop}[thm]{Proposition}
\newtheorem{lem}[thm]{Lemma}
\newtheorem{cor}[thm]{Corollary}
\newtheorem{conj}[thm]{Conjecture}

\def\card{\mathrm{Card}}
\def\A{\mathcal{A}}

\newcommand{\EE}{\mathop{\hbox{\rm I\kern-0.17em E}}\nolimits}
\newcommand{\PP}{\mathop{\hbox{\rm I\kern-0.17em P}}\nolimits}

\newcommand{\E}{{\mathbb E}}
\newcommand{\N}{{\mathbb N}}
\newcommand{\Z}{{\mathbb Z}}
\renewcommand{\P}{{\mathbb P}}

\newcommand{\R}{{\mathbb R}}
\newcommand{\C}{{\mathbb C}}
\newcommand{\T}{{\mathbf T}}
\newcommand{\ind}{{\bf 1}}

\newcommand{\Card}{{\rm Card}}

\def\i{\mathbf{i}}

\def\ie{\textit{i.e.} }

\def\eg{\textit{e.g.} }

\def\sign{\mbox{sign}}

\title{Semi-infinite paths of the 2d-Radial Spanning Tree}

\author{Fran\c{c}ois Baccelli, David Coupier, Viet Chi Tran}

\date{\today}

\begin{document}

\maketitle
\begin{abstract}
We study semi-infinite paths of the radial spanning tree (RST) of a Poisson point
process in the plane. We first show that the expectation of the number of intersection points between semi-infinite paths and the sphere with radius $r$ grows sublinearly with $r$. Then, we prove that in each (deterministic) direction, there exists with probability one a unique semi-infinite path, framed by an infinite number of other semi-infinite paths of close asymptotic directions. The set of (random) directions in which there are more than one semi-infinite paths is dense in $[0,2\pi)$. It corresponds to possible asymptotic directions of competition interfaces. We show that the RST can be decomposed in at most five infinite subtrees directly connected to the root. The interfaces separating these subtrees are studied and simulations are provided.
\end{abstract}

\noindent \textbf{Keywords:} stochastic geometry; random tree; semi-infinite path; asymptotic direction; competition interface 

\noindent \textbf{AMS:} 60D05 

\section{Introduction}

In this paper, we are interested in semi-infinite paths of the 2d-\textit{Radial Spanning Tree} (RST) introduced in \cite{baccellibordenave}. Let us consider a homogeneous Poisson point process (PPP) $N$ on $\mathbb{R}^{2}$ (endowed with its usual Euclidean norm $|.|$) with intensity $1$. Throughout this paper, $N$ is considered in its Palm version: it a.s. contains the origin $O$. The RST is a random graph $\mathcal{T}(N)$ (or merely $\mathcal{T}$) defined as follows. Its vertex set is $N$. Its edge set $E$ contains each pair $\{X,Y\}$, $X,Y\in N$ and $X\not= O$ such that
\begin{equation}
\label{defRST}
|Y| < |X| \; \mbox{ and } \; N \cap B(O,|X|) \cap B(X,|X-Y|) = \emptyset
\end{equation}
(where $B(c,r)$ denotes the open ball with center $c$ and radius $r$). For any $X\in N\setminus\{O\}$, there is a.s. only one $Y\in N$ satisfying (\ref{defRST}). Among the vertices of $N\cap B(O,|X|)$, this is the closest to $X$. This vertex is denoted by $\A(X)$ and called the \textit{ancestor} of $X$. With an abuse of notation, for an edge $e$, we call ancestor of the edge the endpoint that is the ancestor of the other endpoint, which we call the descendant of the edge. The ancestor of $e$ is the endpoint of $e$ that is closer from $O$. With probability $1$, the graph $\mathcal{T}$ admits a tree structure (there is no loop) rooted at the origin $O$. For convenience, we set $\A(O)=O$.

A sequence $(X_{n})_{n\geq 0}$ of vertices of $N$ is a semi-infinite path of the RST $\mathcal{T}$ if, for any $n$, $X_{n}$ is the ancestor of $X_{n+1}$. A semi-infinite path $(X_{n})_{n\geq 0}$ has asymptotic direction $\theta\in [0,2\pi)$ if
$$
\lim_{n\to\infty} \frac{X_{n}}{|X_{n}|} = e^{\i\theta}
$$
(by identifying $\mathbb{R}^{2}$ with the complex plane $\mathbb{C}$).\\

Radial random trees in the plane have been studied in many other papers, including Pimentel \cite{pimentel}, Bonichon and Marckert \cite{bonichonmarckert} or Norris and Turner \cite{norristurner}. The main difficulty in the case of the RST, is that the local rule \eqref{defRST} used for selecting the ancestor implies complex dependences. The latter make it difficult to exhibit natural Markov processes and prevent a direct use of martingale convergence theorems or Lyapunov functions.

However, our study can rely on the following result of \cite[Theorem 2.1]{baccellibordenave} on the RST:

\begin{thm}
\label{HN1}
The following properties hold almost surely:\\
(i) Every semi-infinite path of $\mathcal{T}$ has an asymptotic direction;\\
(ii) For every $\theta\in[0,2\pi)$, there exists at least one semi-infinite path with asymptotic direction $\theta$;\\
(iii) The set of $\theta$'s in $[0,2\pi)$ such that there is more than one semi-infinite path with asymptotic direction $\theta$ is dense in $[0,2\pi)$.
\end{thm}

\noindent
This result is based on a clever method due to Howard and Newman \cite[section 2.3]{howardnewman} proving that the above properties hold for any deterministic tree that satisfies some straightness condition, which is shown to be a.s. satisfied for the RST (\cite[Theorem 5.4]{baccellibordenave}).\\

Understanding finer structural properties of radial random trees,
such as the asymptotic directions of their infinite branches or
the shape of the interfaces that separate subtrees, has been
a recurrent question. See for instance \cite{coupier,coupierheinrich1,coupierheinrich2} for results on the geodesics and interfaces of the last passage percolation tree. The aim of the present paper is to investigate these
questions in the RST.\\

Our first result concerns the number of intersection points between semi-infinite paths of the RST $\mathcal{T}$ and the sphere with radius $r$ centered at the origin; its expectation tends to infinity but slower than $r$ (Theorem \ref{theo:sublin}). The proof is based on the local approximation of the RST, far enough from the origin, by the Directed Spanning Forest (DSF). See \cite{baccellibordenave} for details. Moreover, it has been proved recently (\cite[Theorem 8]{coupiertran}) that there is no bi-infinite path in the DSF. This allows us to conclude.

Subsequently, we focus our attention on semi-infinite paths with deterministic directions. Proposition \ref{prop:<2} states that, for any given $\theta\in [0,2\pi)$, there is almost surely exactly one semi-infinite path with direction $\theta$. Let $\gamma_{0}$ be the one corresponding to $\theta=0$. Proposition \ref{prop:<2} provides a further description of the subtree of the RST made up of $\gamma_{0}$ and all the branches emanating from it.

Finally, we study the subtrees of the RST rooted at the children of $O$. To do it, each of these subtrees is painted with a different color. This process produces the \textit{Colored RST}. Each of these colored subtrees can be bounded or not. Since the origin can have at most 5 descendants with probability 1, there are at most 5 distinct unbounded subtrees rooted at $O$. We prove in Theorem \ref{arbresinfinis} that their number may be equal to $1$, $2$, $3$, $4$ or $5$ with positive probability. The border between two colored subtrees is called \textit{competition interface}. Any unbounded competition interface admits an asymptotic direction (Proposition \ref{prop:cvps}). This direction is random and corresponds to the one of (at least) two semi-infinite paths, as in Part (iii) of Theorem \ref{HN1}.

It is worth pointing out here that our proofs strongly rely on the planarity of the RST and the non-crossing property of its branches (see Lemma \ref{lemm:croisement} in appendix). They cannot be carried to an arbitrary dimension.\\

The paper is organized as follows. In Section \ref{section:sublinear}, the sublinear character of the expected number of intersection points between semi-infinite paths and the sphere with radius $r$ is established. Section \ref{section:directiondeterministe} contains results on semi-infinite paths with deterministic directions. The colored RST and competition intefaces are defined in Section \ref{section:coloredRST}. Finally, open questions and numerical studies are gathered in Section \ref{section:conjectures}.

\section{Sublinearity of the number of semi-infinite paths}
\label{section:sublinear}

Let $r$ be a positive real number. Let us denote by $\chi_{r}$ the number of intersection points of the sphere $S(O,r)=\{r e^{\i \theta},\theta\in [0,2\pi)\}$ with the semi-infinite paths of the RST. The main result of this section states that the expectation of $\chi_{r}$ is sublinear.

\begin{thm}
\label{theo:sublin}
The following limit holds:
$$
\lim_{r\to\infty} \E \Big(\frac{\chi_{r}}{r}\Big) = 0 ~.
$$
\end{thm}

The idea of the proof is as follows. For $r>0$, we introduce the points $A_r=r e^{\i / r}$ and $B_r=r e^{-\i /r}$ of the sphere $S(O,r)$. In the sequel, we will denote by $[A_r,B_r]$ the line segment with extremities $A_r$ and $B_r$ and by $a(A_r, B_r)=\{re^{\i \theta}, \theta\in [-1/r,1/r]\}$ the arc of $S(O,r)$ with extremities $A_r$ and $B_r$ and containing the point $(r,0)$. This arc is by construction of length 2. We denote by $\widetilde{\chi}_{r}$ the number of intersection points between semi-infinite paths of the RST and $a(A_r, B_r)$. By the rotational invariance of the PPP $N$, $\E(\chi_r)=\pi r \ \E(\widetilde{\chi}_r)$. Hence, using an additional moment condition, the proof of Theorem \ref{theo:sublin} will be shown to amount to proving that
\begin{equation}
\label{step1}
\lim_{r\to\infty} \P(\widetilde{\chi}_r\geq 1) = 0 ~.
\end{equation}
To prove \eqref{step1}, notice that far enough from the origin, the RST can be locally approximated by a directed forest named Directed Spanning Forest (DSF, see \cite{baccellibordenave}). The DSF $\mathcal{T}_{-e_x}$ with direction $-e_x=-(1,0)$ is a graph built on the PPP $N$ and in which each vertex $X$ has as ancestor the closest point of $N$ among those with strictly smaller abscissa. This construction generates a family of trees, i.e. a forest, which bears similarities with other directed forests introduced in the literature (see e.g. \cite{athreyaroysarkar,FP,gangopadhyayroysarkar}). \\
As $r$ tends to infinity, the neighborhood of $(r,0)$ in the RST increasingly looks similar to the neighborhood of $O$ in the DSF. Hence, the probability $\P(\widetilde{\chi}_r\geq 1)$ that there exists an infinite path crossing $a(A_r,B_r)$ is close to that of having a path of the DSF crossing $\{0\}\times[-1,1]$ and very long in the direction $e_x$. Such a phenomenon is rare since the DSF is known to have a.s. only one topological end \cite{coupiertran}.\\

In order to prove Theorem \ref{theo:sublin}, we will need the two following lemmas whose proofs are deferred to the end of the section.

\begin{lem}
\label{lemme:moment2}
For any $r>0$, the number of edges of the RST that intersect an arc of $S(O,r)$ of length 1 has finite second order moment and moreover :
$$
\limsup_{r\rightarrow+\infty}\E(\widetilde{\chi}_r^2)<+\infty ~.
$$
\end{lem}

Lemma \ref{lemme:approximationDSF} specifies how the RST is approximated by the DSF: far from the origin (around the point $(r,0)$) and locally (for the neighborhood of radius $r^\alpha$ of $(r,0)$). Notice that since the distribution of the DSF is invariant by translation along $e_x$, $\mathcal{T}_{-e_x}\cap B((r,0),r^\alpha)$ and $\mathcal{T}_{-e_x}\cap B(O,r^\alpha)$ have same distribution.

\begin{lem}
\label{lemme:approximationDSF}
Let $\mathcal{T}$ and $\mathcal{T}_{-e_x}$ respectively denote the RST and DSF of direction $-e_x$ constructed on the same PPP $N$.
Then, for $0<\alpha<1/3$:
\begin{equation}
\lim_{r\rightarrow +\infty} \P\big(\mathcal{T}\cap B((r,0),r^\alpha)=\mathcal{T}_{-e_x}\cap B((r,0),r^\alpha)\big)=1.
\end{equation}
The approximation also holds if we replace $r^\alpha$ by a constant radius $R$.
\end{lem}

\begin{proof}[Proof of Theorem \ref{theo:sublin}]
\textbf{Step 1:} Let us prove \eqref{step1}. First of all, notice that all the paths which intersect the arc $a(A_r, B_r)$ necessarily intersect the segment $[A_r,B_r]$ (the converse is not necessarily true). The segment $[A_r,B_r]$ is perpendicular to the horizontal axis, and all its points have abscissa $\tilde{r}=r \cos(1/ r)$. Its length is $2r \sin(1/r)\leq 2$. \\
Heuristically, the event $\{\widetilde{\chi}_r\geq 1\}$ (consisting in the existence of at least one semi-infinite path crossing $a(A_r, B_r)$) is hence close to the existence of a path of the RST crossing the vertical segment $[A_r ,B_r]$ and then surviving until a large radius.\\
For $R>0$, let us hence consider the event where there exists a path of the RST crossing $ \{\widetilde{r}\}\times [-1,1] \supset [A_r, B_r]$ before intersecting the sphere $S((\tilde{r},0),R)$.\\
Our purpose is to show that the probability of this event is close to the probability that in the DSF $\mathcal{T}_{-e_x}$, there exists a path intersecting $\{0\}\times [-1,1]$ and then $S(O,R)\cap \{x>0\}$. To show that such an approximation holds, let us prove that our event is local in the sense of Lemma \ref{lemme:approximationDSF}.\\

\begin{figure}[!ht]
\begin{center}
\psfrag{a}{\footnotesize{$Z_{0}$}}
\psfrag{b}{\footnotesize{$Z_{1}$}}
\psfrag{c}{\footnotesize{$Z_{n\!-\!1}$}}
\psfrag{d}{\footnotesize{$Z_{n}$}}
\psfrag{e}{\footnotesize{$S((\tilde{r},0),R)$}}
\psfrag{f}{\tiny{$-1$}}
\psfrag{g}{\tiny{$+1$}}
\includegraphics[width=7cm,height=6.5cm]{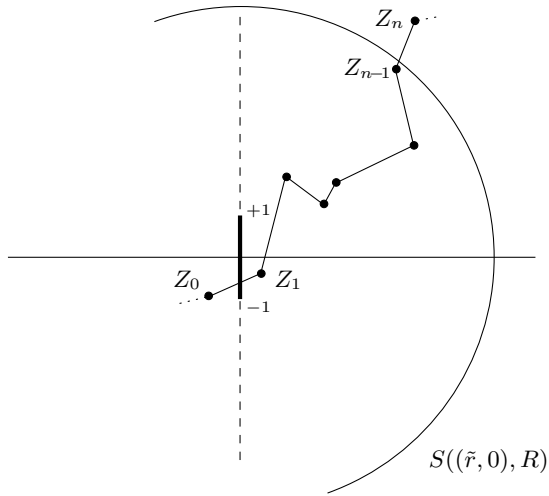}
\end{center}
\caption{\label{fig:Zn} {\small \textit{Here is the sub-path $Z_{0},\ldots,Z_{n}$ of the RST crossing the vertical segment $\{\widetilde{r}\}\times [-1,1]$ (in bold) and the sphere $S((\tilde{r},0),R)$. On this picture, $Z_{0}$ and $Z_{n}$ belong to $B((\widetilde{r},0),2R)$ which occurs with high probability.}}}
\end{figure}

Let us consider a path of the RST crossing $\{\widetilde{r}\}\times [-1,1]$ and afterwards $S((\tilde{r},0),R)$, i.e. towards descendants as is described below. From this path, we can extract a sub-path $(Z_{0},\ldots,Z_{n})$ crossing only once $\{\widetilde{r}\}\times [-1,1]$ (between $Z_{0}$ and $Z_{1}$) and $S((\tilde{r},0),R)$ (between $Z_{n-1}$ and $Z_{n}$). See Figure \ref{fig:Zn}. We show that this path is included in the ball $B((\widetilde{r},0),2R)$, so that the local approximation of the RST by the DSF (see Lemma \ref{lemme:approximationDSF}) holds. If $Z_{0}$ is outside the ball $B((\tilde{r},0),R)$, then $B(Z_{1},|Z_{1}-Z_{0}|)$ contains $B((\tilde{r},0),R/2)$ for $R$ large enough. Consequently, the set $B(O,\tilde{r})\cap B((\tilde{r},0),R/2)$ is empty of points of the PPP $N$. So, $Z_{0}$ belongs to $B((\tilde{r},0),R)$ with high probability as $r,R$ tend to infinity (with $r$ tending to infinity faster than $R$). The same is true about the endpoint $Z_{n}$ and the ball $B((\tilde{r},0),2R)$. To sum up, given $\varepsilon>0$ and $r,R$ large enough,
\begin{align}
\P( \widetilde{\chi}_r \geq 1 )  \leq & \; \P \left(\begin{array}{c}
\mbox{there exists a path of the RST crossing} \\
\mbox{$\{\tilde{r}\}\times[-1;1]$ and afterwards $S((\tilde{r},0),R)$,} \\
\mbox{whose endpoints belong to $B((\tilde{r},0),2R)$.}
\end{array}\right) + \varepsilon \nonumber\\
  \leq & \; \P \left(\begin{array}{c}
\mbox{there exists a path of the DSF crossing} \\
\mbox{$\{0\}\times[-1;1]$ and afterwards $S(O,R)$}
\end{array}\right) + 2 \varepsilon ~.\label{DSFcontresens}
\end{align}
Theorem 8 of \cite{coupiertran} says each path of the DSF is a.s. finite towards descendants. Then, the probability in the right-hand side of (\ref{DSFcontresens}) tends to $0$ as $R$ tends to infinity. This means that $\P(\widetilde{\chi}_r\geq 1)$ is smaller than $3\varepsilon$, i.e. (\ref{step1}).\\

\textbf{Step 2:} Let us prove that $\E(\chi_{r})$ is sublinear. As mentioned at the beginning of the section, it is sufficient to show $
\lim_{r\to\infty} \E (\widetilde{\chi}_r) = 0$.\\
Since by the Cauchy-Schwarz inequality:
$$
\E\big( \widetilde{\chi}_r\big) \leq \sqrt{\E \big(\widetilde{\chi}_r^{2}\big)} \sqrt{\P(\widetilde{\chi}_{r}\geq 1)} ~,
$$
the desired limit follows from (\ref{step1}) if we prove that $\limsup_{r\rightarrow+\infty}\E(\widetilde{\chi}_{r}^{2})$ is finite, which is the result of Lemma \ref{lemme:moment2}.
\hfill $\Box$ \end{proof}

The section ends with the proofs of Lemmas \ref{lemme:moment2} and \ref{lemme:approximationDSF}.

\begin{proof}[Proof of Lemma \ref{lemme:moment2}]
Let $A_r$ and $B_r$ be as in the proof of Theorem \ref{theo:sublin}. Let $W_r$ be the intersection point of the two tangents to $S(O,r)$ that pass through the points $A_r$ and $B_r$. See the left part of Figure \ref{fig:arc}.

\begin{figure}[!ht]
\begin{center}
\psfrag{O}{\footnotesize{$O$}}
\psfrag{A}{\footnotesize{$A_r$}}
\psfrag{B}{\footnotesize{$B_r$}}
\psfrag{Z}{\footnotesize{$W_r$}}
\psfrag{S}{\footnotesize{$S(O,r)$}}
\psfrag{a}{\scriptsize{$\mathcal{A}(X)$}}
\psfrag{b}{\footnotesize{$(r,0)$}}
\psfrag{c}{\scriptsize{$J(X)$}}
\psfrag{d}{\footnotesize{$I_{r}\oplus B(0,c)$}}
\psfrag{e}{\footnotesize{$z(X)$}}
\psfrag{f}{\footnotesize{$X$}}
\begin{tabular}{cp{1cm}c}
\includegraphics[width=5.5cm,height=5cm]{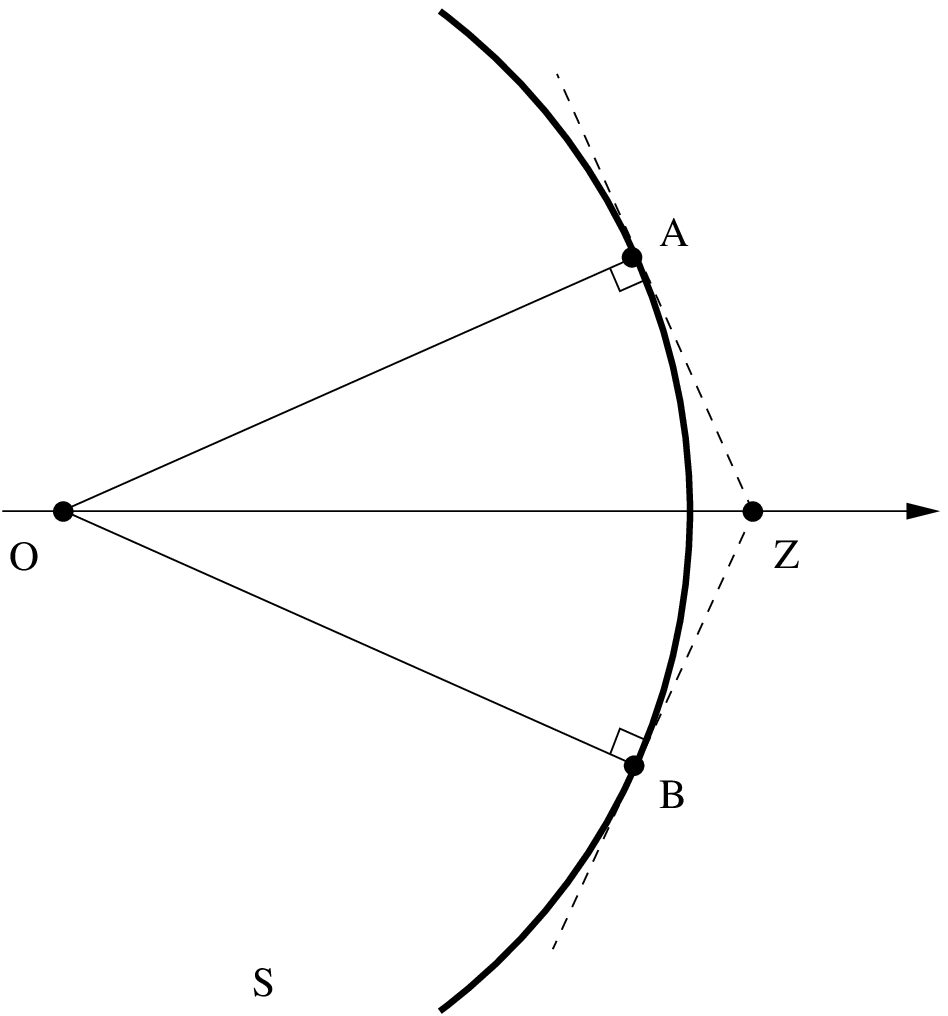} & &
\includegraphics[width=6.5cm,height=5cm]{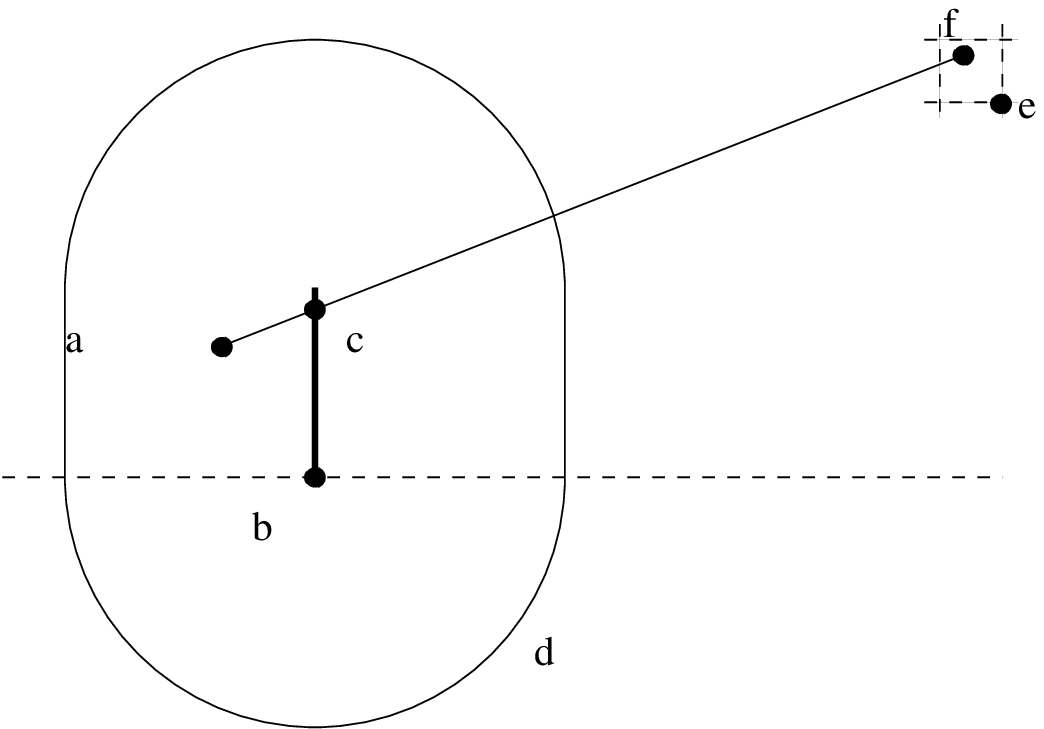}
\end{tabular}
\end{center}
\caption{\label{fig:arc} {\small \textit{On the left: The sphere $S(O,r)$ with center $O$ and radius $r$ is represented in bold. The tangents (the dotted lines) to $S(O,r)$ respectively at $A_r$ and $B_r$, intersect the horizontal axis on $W_r$. On the right: The segment $[X,\mathcal{A}(X)]$ crosses $I_{r}$ (in bold) on $J(X)$. On this picture, $X$ is outside $I_{r}\oplus B(0,c)$.}}}
\end{figure}


The number of semi-infinite paths that cross $a(A_r,B_r)$ is upper bounded by the number $\check{\chi}_r+\widehat{\chi}_r$ of edges of the RST which intersect $[A_r,W_r]\cup [W_r,B_r]$, where $\check{\chi}_r$ (resp. $\widehat{\chi}_r$) denotes the number of edges crossing $[A_r,W_r]$ (resp. $[W_r,B_r]$) and whose ancestors belong to the same half plane delimited by the line supporting $[A_r,W_r]$ (resp. $[W_r,B_r]$) as $O$. Since $\check{\chi}_r$ and $\widehat{\chi}_r$ are identically distributed:
\begin{equation}
\E(\widetilde{\chi}_{r}^{2})\leq \E((\check{\chi}_{r}+\widehat{\chi}_r)^{2})\leq 4 \E(\check{\chi}_r^2),\label{etape5}
\end{equation}
and it is sufficient to show that $\limsup_{r\to\infty}\E(\check{\chi}_r^2)$ is finite. By rotational invariance, the distribution of $\check{\chi}_r$ is also the distribution of the number of edges with ancestors of smaller abscissa and that cross the vertical segment $I_r=[(r,r\tan(1/r)),(r,0)]$. With an abuse of notation, we will denote again by $\check{\chi}_r$ the last random variable. Notice also that the length of $I_r$ is bounded by 2 as soon as $r$ is sufficiently large. Let $c>2+3\sqrt{2}$ (a technical condition needed in the sequel), and let us use $\oplus$ for the Minkowski addition. Then:
\begin{equation}
\check{\chi}_r=\check{\chi}_r^{\leq c}+\check{\chi}_r^{>c} ~\mbox{ a.s. } ,\label{decomp}
\end{equation}
where $\check{\chi}_r^{\leq c}$ (resp. $\check{\chi}_r^{>c}$) denotes the number of these edges with descendants belonging to $I_r\oplus B(O,c)$ (resp. being at a distance at least $c$ from $I_r$). $\check{\chi}_r^{\leq c}$ is upper bounded by $\card(N\cap (I_r\oplus B(O,c)))$ and admits a moment of order 2 that is bounded independently of $r$. It remains to study $\check{\chi}_r^{>c}$. Our idea is that each long edge is accompanied by a large empty space, so that it is rare that many long edges intersect $I_r$.\\
Let us consider an edge $[\mathcal{A}(X),X]$ that crosses $I_r$ at $J(X)$ and such that the distance from $X$ to $I_r$ is larger than $c$. If $X=(x,y)$, then let us consider $z(X)$ the point with coordinates $([x]+1,\sign(y)[|y|])$, where $[x]$ denotes the integer part of $x$. Among the points with integer coordinates that have an abscissa larger than $x$ and which are closer to the abscissa axis than $X$, $z(X)$ is the point that is the closest to $X$. See the right part of Figure \ref{fig:arc}.\\
By construction, $|X-z(X)|\leq \sqrt{2}$. Hence $B(O,|z(X)|-\sqrt{2})\subset B(O,|X|)$, where the radius of the first ball is positive as soon as $r\geq \sqrt{2}$. Let us consider the ball $B(z(X),|z(X)-(r,0)|-2-2\sqrt{2})$. For our choice of $c$, $|z(X)-(r,0)|-2-2\sqrt{2}\geq c-\sqrt{2}-2-2\sqrt{2}\geq 0$ and the radius is positive. If $U\in B(z(X),|z(X)-(r,0)|-2-2\sqrt{2})$, then
\begin{equation}
|U-X|\leq |U-z(X)|+\sqrt{2}\leq |z(X)-(r,0)|-2-\sqrt{2}\leq |X-J(X)|\leq |X-\mathcal{A}(X)|,
\end{equation}and thus $B(z(X),|z(X)-(r,0)|-2-2\sqrt{2})\subset B(X,|X-\mathcal{A}(X)|)$. As a consequence, if we introduce $\Lambda(z,r)=B(O,|z|)\cap B(z,|z-(r,0)|-2-2\sqrt{2})$ for $z=(z_x,z_y)\in \Z^2$ and $r$ sufficiently large, we have that $\card(N\cap \Lambda(z(X),r))\leq \card(N\cap B(O,|X|)\cap B(X,|X-\mathcal{A}(X)|)),$ the latter quantity being 0 since $[\mathcal{A}(X),X]$ is an edge of the RST, implying that there is no point of $N$ in $\Lambda(z(X),r)$. Thus, for $r\geq \sqrt{2}$:
\begin{equation}
\check{X}^{>c}_r\leq Y_r:= \sum_{\substack{z=(z_x,z_y)\in \Z^2 \\ z_x\geq r}} \ind_{\{\Lambda(z,r)\cap N=\emptyset\}}, ~ \mbox{ a.s.}\label{2.7}
\end{equation}Notice that if $r$ and $r'$ are such that $z_x\geq r'\geq r$, then $\Lambda(z,r)\subset \Lambda(z,r')$. This implies that $r\mapsto Y_r$ is almost surely a decreasing function of $r$. Then, if we fix $r_0 \geq \sqrt{2}$, $\forall r\geq r_0$, $\check{X}^{>c}_r\leq Y_{r_0}$. The volume of $\Lambda(z,r_0)$ is of the order of $|z|^2$ and for a given integer $\rho\geq r_0^2$, the number of points $z$ such that $|z|^2=\rho$ is of the order of $\sqrt{\rho}$. Thus, for two positive constants $C$ and $C'$:
\begin{align}
\E\big(Y_{r_0}\big)=\sum_{\substack{z=(z_x,z_y)\in \Z^2 \\ z_x\geq r_0}} \P\big(\Lambda(z,r_0)\cap N=\emptyset\big)\leq C\sum_{\rho\geq r_0^2} \sqrt{\rho} e^{-C'\rho}<+\infty.\label{etape3}
\end{align}It now remains to prove that $\E(Y_{r_0}^2)<+\infty$. For this, we compute
\begin{multline}
\E\Big(Y_{r_0}(Y_{r_0}-1)\Big)=  \sum_{\substack{z=(z_x,z_y)\in \Z^2,\\ z'=(z'_x,z'_y)\in \Z^2 \\ z_x\geq r_0, \ z'_x\geq r_0\\ z_x\not=z'_x}} \E\left(\ind_{\{\Lambda(z,r_0)\cap N=\emptyset\}}\ind_{\{\Lambda(z',r_0)\cap N=\emptyset\}}\right)\\
 \leq  2 \sum_{\rho\geq r_0^2} \sum_{\substack{z=(z_x,z_y)\in \Z^2,\\ |z_x|^2=\rho }} \sum_{\substack{z'=(z'_x,z'_y)\in \Z^2,\\ z'_x\geq r_0 \\ |z'|\leq |z|}} \P\Big(\Lambda(z,r_0)\cap N=\emptyset\Big)
 \leq  C \sum_{\rho\geq r_0^2} \rho^{3/2}e^{-C'\rho}<+\infty\label{etape4}
\end{multline}for two positive constants $C$ and $C'$. \eqref{etape3} and \eqref{etape4} show that $\E(Y_{r_0}^2)<+\infty$ and this concludes the proof.
\hfill $\Box$ \end{proof}

\begin{proof}[Proof of Lemma \ref{lemme:approximationDSF}]
We follow here the proof of Baccelli and Bordenave \cite[Section 3.6]{baccellibordenave}, where the case of a fixed radius $R$ is considered. Recall that $\mathcal{T}$ and $\mathcal{T}_{-e_x}$ are the RST and DSF with direction $-e_x$, constructed on the same PPP $N$. We denote by $\mathcal{A}(X)$ and $\mathcal{A}_{-e_x}(X)$ the ancestors of $X$ in $\mathcal{T}$ and $\mathcal{T}_{-e_x}$. Let $r>0$, $\alpha>0$ and $\beta>0$.
\begin{multline*}
\P\big(\mathcal{T}\cap B((r,0),r^\alpha) \neq \mathcal{T}_{-e_x}\cap B((r,0),r^\alpha)\big)=  \P\Big(\bigcup_{X\in N\cap B((r,0),r^\alpha)}\big\{\mathcal{A}(X)\not=\mathcal{A}_{-e_x}(X)\big\}\Big)\\
\begin{aligned}
\leq & \P\big(N(B((r,0),r^\alpha))> r^\beta\big)+r^\beta \ C \sup_{X\in N\cap B((r,0),r^\alpha)} \P\big(\mathcal{A}(X)\not=\mathcal{A}_{-e_x}(X)\big) \nonumber\\
\leq & \exp\Big(-r^\beta \log \big(\frac{r^{\beta-2\alpha}}{e \pi}\big)\Big)+r^\beta \ C \ \frac{r^\alpha+1}{r-r^\alpha},
\end{aligned}
\end{multline*}by using \cite[Lemma 11.1.1]{talagrand} for the first term in the r.h.s. and \cite[Lemma 3.4]{baccellibordenave} for the second term. The first term converges to 0 iff $\beta>2\alpha$ and the second term converges to 0 iff $\alpha<1$ and $\beta+\alpha<1$. As a consequence, we see that for any $\alpha<1/3$ we can choose $\beta>2/3$, so that both terms converges to 0 when $r\rightarrow +\infty$.
\hfill $\Box$ \end{proof}

\section{Semi-infinite paths in a given direction}\label{section:directiondeterministe}

In this section, we fix a direction $\theta\in [0,2\pi)$ and are interested in the semi-infinite paths with asymptotic direction $\theta$. Our first result (Section \ref{section:unique}) refines Theorem \ref{HN1} and states that there exists a.s. a unique semi-infinite path with direction $\theta$. We deduce from this a precise description of the semi-infinite path with direction $\theta$ (Section \ref{section:description0}).\\

For the proof, let us introduce further notation. We define as $\mathcal{T}_{X}$ the subtree of $\mathcal{T}$ consisting of $X\not=O$ and all its descendants, i.e. all the vertices of $\mathcal{T}$ that have $X$ in their ancestry. This tree is naturally rooted at $X$.\\
If $\mathcal{T}_X$ is unbounded, then we can construct two particular semi-infinite paths that we call the right-most and left-most semi-infinite paths, $\underline{\gamma}_X$ and $\overline{\gamma}_X$, of $\mathcal{T}_X$. The construction follows.\\
Put $X_{0}=X$. Let
$$K_0=\Card\{Y\in N,\ \mathcal{A}(Y)=X_0\mbox{ and }\mathcal{T}_Y\mbox{ is unbounded }\}$$
be the number children of $X_0$ with infinite descendance. Since the number of children of a given vertex is a.s. finite (see \cite[Section 3.3.2.]{baccellibordenave}) and since $X_{0}$ has infinitely many descendants, $K_0\geq 1$ a.s. It is possible to rank these offspring $X_0^1,\dots,X_0^{K_0}$ by increasing order of the oriented angles $\widehat{\mathcal{A}(X_0)X_0 X_0^k}$ for $k\in \{1,\dots,K_0\}$. Define $X_{1}$ as the child of $X_{0}$ corresponding to the largest value of these angles. Iterating this construction, a semi-infinite path $\overline{\gamma}_{X}=(X_{n})_{n\in\mathbb{N}}$ rooted at $X$ is built.\\
In the same way, a semi-infinite path $\underline{\gamma}_{X}$ rooted at $X$ is constructed such that, among the semi-infinite paths of $\mathcal{T}_{X}$, $\underline{\gamma}_{X}$ is the lowest one (in the trigonometric sense). Consequently, any given semi-infinite path in $\mathcal{T}_{X}$ is trapped between $\underline{\gamma}_{X}$ and $\overline{\gamma}_{X}$ (in the trigonometric sense).

\subsection{Uniqueness}
\label{section:unique}

Part $(iii)$ of Theorem \ref{HN1} ensures the existence of random directions with at least two semi-infinite paths. However, there is no more than one semi-infinite path with a deterministic direction (Proposition \ref{prop:<2}). This result completes Part $(ii)$ of Theorem \ref{HN1}.

\begin{prop}
\label{prop:<2}
For all $\theta\in[0,2\pi)$, there a.s. exists exactly one semi-infinite path with asymptotic direction $\theta$ in the RST.
\end{prop}

The idea of the proof of Proposition \ref{prop:<2} is classical: see \cite{howardnewman2} for first passage percolation models defined from homogeneous PPP on $\mathbb{R}^{2}$ and \cite{FP} for a directed last passage percolation model on the lattice $\mathbb{Z}^{2}$. Thanks to Fubini's theorem, we get that for Lebesgue almost every $\theta$ in $[0,2\pi)$, there is at most one semi-infinite path with asymptotic direction $\theta$ with probability $1$. Actually, this statement holds for all $\theta\in[0,2\pi)$ by the isotropic character of the PPP $N$.

\begin{proof}
Let us denote by $U(\theta)$ the event that there exist at least two different semi-infinite paths in the RST with asymptotic direction $\theta$. Now, assume the event $U(\theta)$ occurs and let $\gamma_{1}$ and $\gamma_{2}$ be two such semi-infinite paths. Let $X$ be a point of the PPP $N$ belonging to $\gamma_{1}$ but not to $\gamma_{2}$. Thus the semi-infinite sub-path of $\gamma_1$ rooted at $X$ belongs to $\mathcal{T}_X$. Then, one of the two semi-infinite paths $\underline{\gamma}_{X}$ and $\overline{\gamma}_{X}$ is trapped between $\gamma_{1}$ and $\gamma_{2}$, by planarity and since paths are non-intersecting (see Lemma \ref{lemm:croisement} in appendix). So, it also admits $\theta$ as asymptotic direction.\\
Let us denote by $\lambda$ the Lebesgue measure on $[0,2\pi)$. We are interested in the Lebesgue measure of the set $\{\theta ; U(\theta) \}$ of directions $\theta\in [0,2\pi)$ where the event $U(\theta)$ is satisfied. The previous remark implies:
\begin{eqnarray*}
\E \lambda \{\theta ; U(\theta) \} & = & \int_{\Omega} \int_{0}^{2\pi} \ind_{U(\theta)}(\omega) \; d\theta \; d\P(\omega) \\
& \leq & \int_{\Omega} \sum_{X\in N(\omega)} \ind_{{\scriptstyle \mathcal{T}_X\;\mbox{\small{unbounded}}}} \int_{0}^{2\pi} \ind_{{\scriptstyle \underline{\gamma}_{X} \;\mbox{\small{or}}\; \overline{\gamma}_{X} \;\mbox{\small{admits}}\; \theta \;\mbox{\small{as}}} \atop \scriptstyle \mbox{\small{asymptotic direction}}}(\omega) \; d\theta \; d\P(\omega)
\end{eqnarray*}
($N$ is a.s. countable). For a given point $X\in N$ and a given $\omega\in \Omega$, the indicator function
$$
\ind_{{\scriptstyle \underline{\gamma}_{X} \;\mbox{\small{or}}\; \overline{\gamma}_{X} \;\mbox{\small{admits}}\; \theta \;\mbox{\small{as}}} \atop \scriptstyle \mbox{\small{asymptotic direction}}}(\omega)
$$
is equal to $1$ for at most two different angles in $[0,2\pi)$. Its integral is then equal to zero. Using the Fubini's theorem,
$$
\int_{0}^{2\pi} \P (U(\theta)) \; d\theta = \E \lambda \{\theta ; U(\theta) \} = 0 ~.
$$
So, the probability $\P(U(\theta))$ is zero for Lebesgue a.e. $\theta$ in $[0,2\pi)$. Actually, this is true for every $\theta$ in $[0,2\pi)$ thanks to the isotropic character of the PPP $N$. Combining this with Part $(ii)$ of Theorem \ref{HN1}, the announced result follows.
\hfill $\Box$ \end{proof}

\subsection{Further description of the semi-infinite path with direction 0}
\label{section:description0}

In the rest of this section, we discuss some consequences of Proposition \ref{prop:<2}.
For any given $\theta\in [0,2\pi)$, let us denote by $\gamma_{\theta}$ the semi-infinite path of the RST, started at the origin and with asymptotic direction $\theta$. It is a.s. well defined by Proposition \ref{prop:<2}. Since the distribution of the RST is invariant by rotation, we will henceforth assume that $\theta=0$.\\

Let us recall that $\widetilde{\chi}_r$ denotes the number of intersection points of $a(A_r, B_r)$ with the semi-infinite paths of the RST and that $\widetilde{\chi}_r\rightarrow 0$ in probability, by \eqref{step1}. We have

\begin{cor}
$\limsup_{r\rightarrow \infty}\widetilde{\chi}_r\geq 1$ a.s.
\end{cor}

\begin{proof}
Assume that there exists with positive probability a (random) radius $r_{0}$ such that $\widetilde{\chi}_r=0$ whenever $r>r_{0}$. Let us work on the set where this event is realized. In this case, no semi-infinite path crosses the abscissa axis after $r_0$. Then, we can exhibit two semi-infinite paths, say $\gamma$ and $\gamma'$, respectively below and above the horizontal axis, and satisfying the following property: there is no semi-infinite path in the RST, different from $\gamma$ and $\gamma'$, and trapped between them (in the trigonometric sense). Parts $(i)$ and $(ii)$ of Theorem \ref{HN1} force $\gamma$ and $\gamma'$ to have the same asymptotic direction, namely $0$. Such a situation never happens by Proposition \ref{prop:<2}. In other words,
$$
\P \left( \limsup_{r\to\infty} \widetilde{\chi}_r \geq 1 \right) \, = \, 1 ~.
$$
\hfill $\Box$ \end{proof}

From vertices of $\gamma_{0}$ (different from $O$), some paths (finite or not) emanate, forming together an unbounded subtree of the RST $\mathcal{T}$ for which $\gamma_{0}$ can be understood as the spine. The next results describe the skeleton of this subtree.\\
Let us denote by $V_{\infty}^{+}$ and $V_{\infty}^{-}$ the set of points $X\in N\cap \gamma_{0}\setminus\{O\}$ from which (at least) another semi-infinite path emanates, respectively above and below $\gamma_{0}$. Of course, $V_{\infty}^{+}$ and $V_{\infty}^{-}$ may have a nonempty intersection.

\begin{cor}
\label{corol:skeleton}
\begin{enumerate}
\item Almost surely, $V_{\infty}^{+}$ and $V_{\infty}^{-}$ are of infinite cardinality.
\item For $r>0$, let us denote by $D_r$ the set of directions $\alpha\in[0,2\pi)$ with a semi-infinite path starting from a point $X$ in $V_{\infty}^{+}\cup V_{\infty}^{-}$ with modulus $|X|>r$. Then, there a.s. exist two nonincreasing sequences $(\alpha_{r})_{r>0}$ and $(\beta_{r})_{r>0}$ of positive r.v.'s such that
$$
D_r = [-\alpha_{r},\beta_{r}] \; \mbox{ (modulo $2\pi$) and } \;  \lim_{r\to+\infty} \alpha_{r} = \lim_{r\to+\infty} \beta_{r} = 0 ~.
$$
\item Let $v_{\infty}^{r}$ be the cardinality of $( V_{\infty}^{+}\cup V_{\infty}^{-} )\cap B(O,r)$. Then,
$$
\lim_{r\to\infty} \E \frac{v_{\infty}^{r}}{r} = 0 ~.
$$
\end{enumerate}
\end{cor}

The first two assertions of Corollary \ref{corol:skeleton} say that an infinite number of unbounded subtrees emanate from the semi-infinite path $\gamma_{0}$. Each of them covers a whole interval of asymptotic directions whose length tends to $0$ as its starting point (on $\gamma_{0}$) is far from the origin. The last assertion of Corollary \ref{corol:skeleton} can be understood as follows: the expected density of points of $\gamma_{0}$ from which emanates another semi-infinite path is zero. For this purpose, recall that the cardinality of $\gamma_{0}\cap B(O,r)$ is of order $r$ (see \cite[Theorem 2.5]{baccellibordenave}).

\begin{proof}
On the event $\{V_{\infty}^{+}\mbox{ is finite}\}$, let us consider the point $Y$ of $V_{\infty}^{+}$ of highest modulus. Let $Y''$ be the child of $Y$ belonging to $\gamma_0$. From the definition of $V_\infty^+$, the set of points $\{X\in N,\ \mathcal{A}(X)=Y,\ \widehat{Y'' Y X}>0 \mbox{ and }\mathcal{T}_X\mbox{ is unbounded}\}$ is non empty. The points of this set can be ranked by increasing values of $\widehat{Y'' Y X}$s. Let $Y'$ be the point of this set corresponding to the smallest positive angle $\widehat{Y'' Y X}$. In the subtree $\mathcal{T}_{Y'}$, we can define the lowest semi-infinite path $\underline{\gamma}_{Y'}$ as defined in the beginning of Section \ref{section:directiondeterministe}. Thanks to Part $(i)$ of Theorem \ref{HN1}, $\underline{\gamma}_{Y'}$ has an asymptotic direction, say $\theta$. Moreover, using the planarity and  the non-crossing property of paths together with the definition of points $Y$, $Y'$, all the paths between $\gamma_0$ and $\underline{\gamma}_{Y'}$ (in the trigonometric sense) are finite. Using Part $(ii)$ of Theorem \ref{HN1}, we deduce that $\underline{\gamma}_{Y'}$ and $\gamma_0$ have the same asymptotic direction, i.e. $\theta=0$. Now, by Proposition \ref{prop:<2}, such a situation never occurs. So the set $V_{\infty}^{+}$ is a.s. infinite. The same goes for $V_{\infty}^{-}$.\\

Let us prove the second part of Corollary \ref{corol:skeleton}. For any $r$, let us consider the point $X$ of smallest modulus among the points of $\gamma_0\cap B(O,r)^{c}$. The semi-infinite paths $\underline{\gamma}_{X}$ and $\overline{\gamma}_{X}$ have a.s. asymptotic directions, say respectively $-\alpha_{r}$ and $\beta_{r}$ (modulo $2\pi$) with $\alpha_{r}\geq 0$, $\beta_{r}\geq 0$. Hence, $(\alpha_{r})_{r>0}$ and $(\beta_{r})_{r>0}$ are by construction non-increasing sequences of positive real numbers (by Proposition \ref{prop:<2}).\\
Let us consider $D_r$ the set of directions corresponding to semi-infinite paths starting from points in $V_{\infty}^{+}\cup V_{\infty}^{-}$ with modulus greater than $r$. $D_r$ contains $\alpha_r$ and $\beta_r$ defined above. The bi-infinite path obtained by concatenation of $\underline{\gamma}_{X}$ and $\overline{\gamma}_{X}$ divides $\R^{2}$ into two unbounded regions. Since the paths of the RST cannot cross, $D_r$ is included in the real interval $[-\alpha_{r},\beta_{r}]$. This also forces any given semi-infinite path with asymptotic direction $\alpha$ in $[-\alpha_{r},\beta_{r}]$ to go through the vertex $X$. By Part $(ii)$ of Theorem \ref{HN1}, $D_r$ is then an interval. It follows $D_{r}=[-\alpha_{r},\beta_{r}]$.\\
Finally, let us respectively denote by $\bar{\alpha}$ and $\bar{\beta}$ the limits of sequences $(\alpha_{r})_{r>0}$ and $(\beta_{r})_{r>0}$. Let $\beta>0$. Let $X_{0,\beta}$ be the bifurcation point of $\gamma_0$ and the semi-infinite path with asymptotic direction $\beta$ (whose existence and uniqueness are given by Proposition \ref{prop:<2}). Then, $\bar{\beta}\geq \beta$ implies that $\beta$ belongs to any interval $D_{r}$. In other words,
$$
\P (\bar{\beta}\geq \beta) = \lim_{r\to\infty} \P (|X_{0,\beta}|>r) = 0 ~.
$$
As a consequence, $\bar{\beta}$ is a.s. equal to $0$. Similarly, we can prove $\bar{\alpha}=0$.

The third part of Corollary \ref{corol:skeleton} follows directly from the inequality $v_{\infty}^{r}\leq \chi_{r}$, where $\chi_{r}$ counts the intersection points of $S(O,r)$ with the semi-infinite paths of the RST, and Theorem \ref{theo:sublin}.
\hfill $\Box$ \end{proof}

\section{The Colored RST}
\label{section:coloredRST}

The aim of this section is to describe the subtrees of the RST $\mathcal{T}$ rooted at the children of $O$. In order to distinguish them, let us start with allocating a color or a label (denoted by integers $i,j,k\ldots$) to each child $X$ of $O$. Recall for this purpose there are a.s. at most $5$ (see Lemma 3.2 in \cite{baccellibordenave}). Then, we paint all the vertices of the subtrees $\mathcal{T}_{X}$ with the color of $X$. This process provides a coloration of points of $N\setminus \{O\}$. Finally, each segment $[X,\A(X)]$, for $X\in N\setminus \{O\}$, is painted with the color of $X$. This can be done without ambiguity thanks to the planarity and the non-crossing property of paths. See Figure \ref{simu_fb}.\\
There are several ways to label the subtrees of $\mathcal{T}$ rooted at $O$ with the colors. We use the notation $i$, $j$, etc. when the labeling is not important and can be any labeling (a possibility among others is to start from the angle 0 (abscissa axis) and label by 1 the subtree rooted at the first descendant of $O$ that we encounter when exploring the directions in the trigonometric sense).\\
However, to take advantage of the exchangeability of the different colored trees, one may also proceed as follows. To each of the direct descendants of $O$, a uniform independent r.v. is attached. We then define the tree with color 1 as the tree consisting of the offspring of the descendant with the smallest uniform r.v. This amounts to choosing one of the descendants at random for the first tree. When we proceed so, we use the notation $\underline{1}, \dots \underline{i},\dots $ for the labels.\\

\begin{figure}[ht]
\begin{center}
\begin{tabular}{ccc}
(a) & (b) & (c) \\
\includegraphics[width = 0.3\textwidth,trim=0cm 0cm 0cm 0cm]{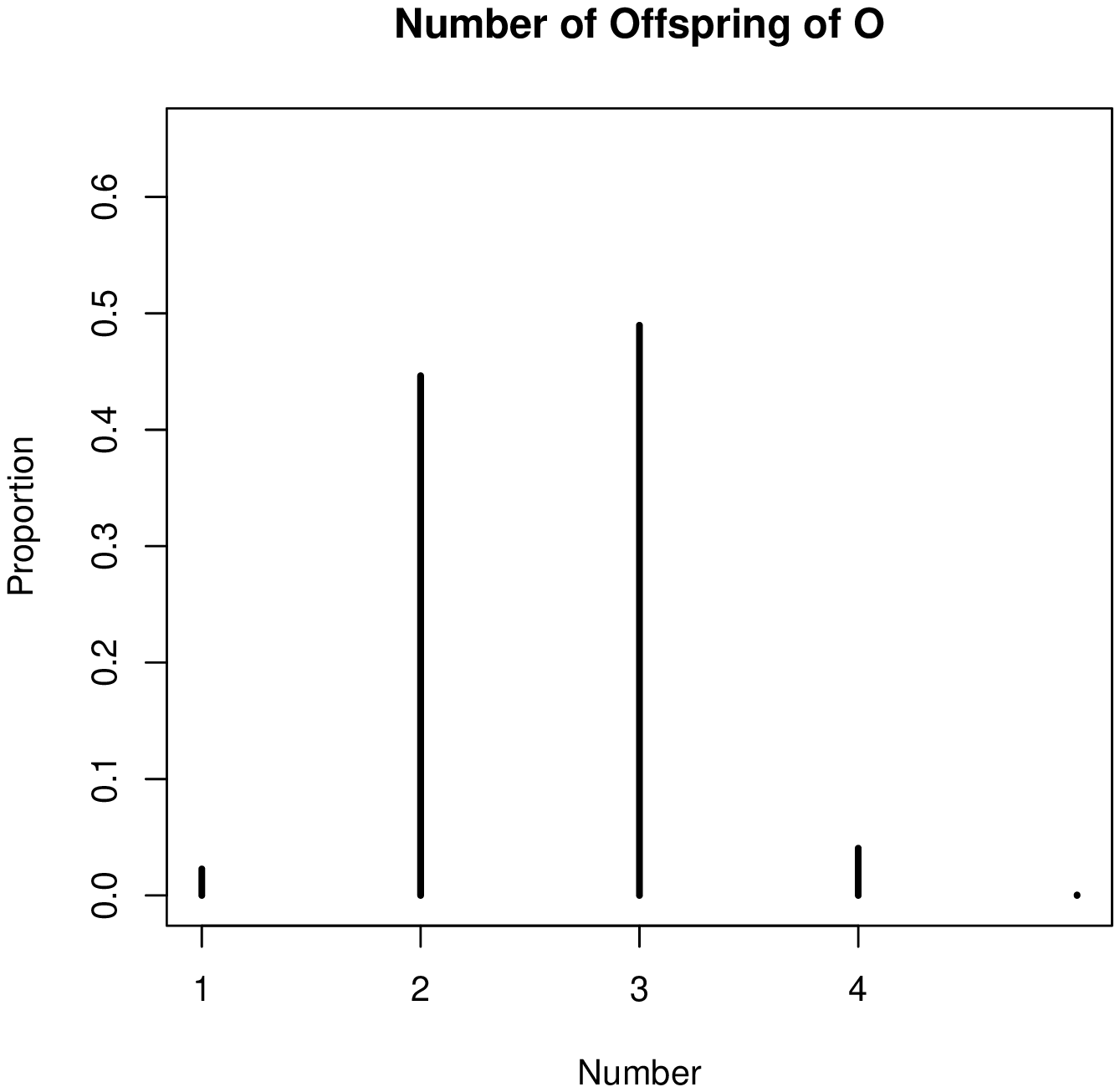} &
\includegraphics[width = 0.3\textwidth,trim=0cm 0cm 0cm 0cm]{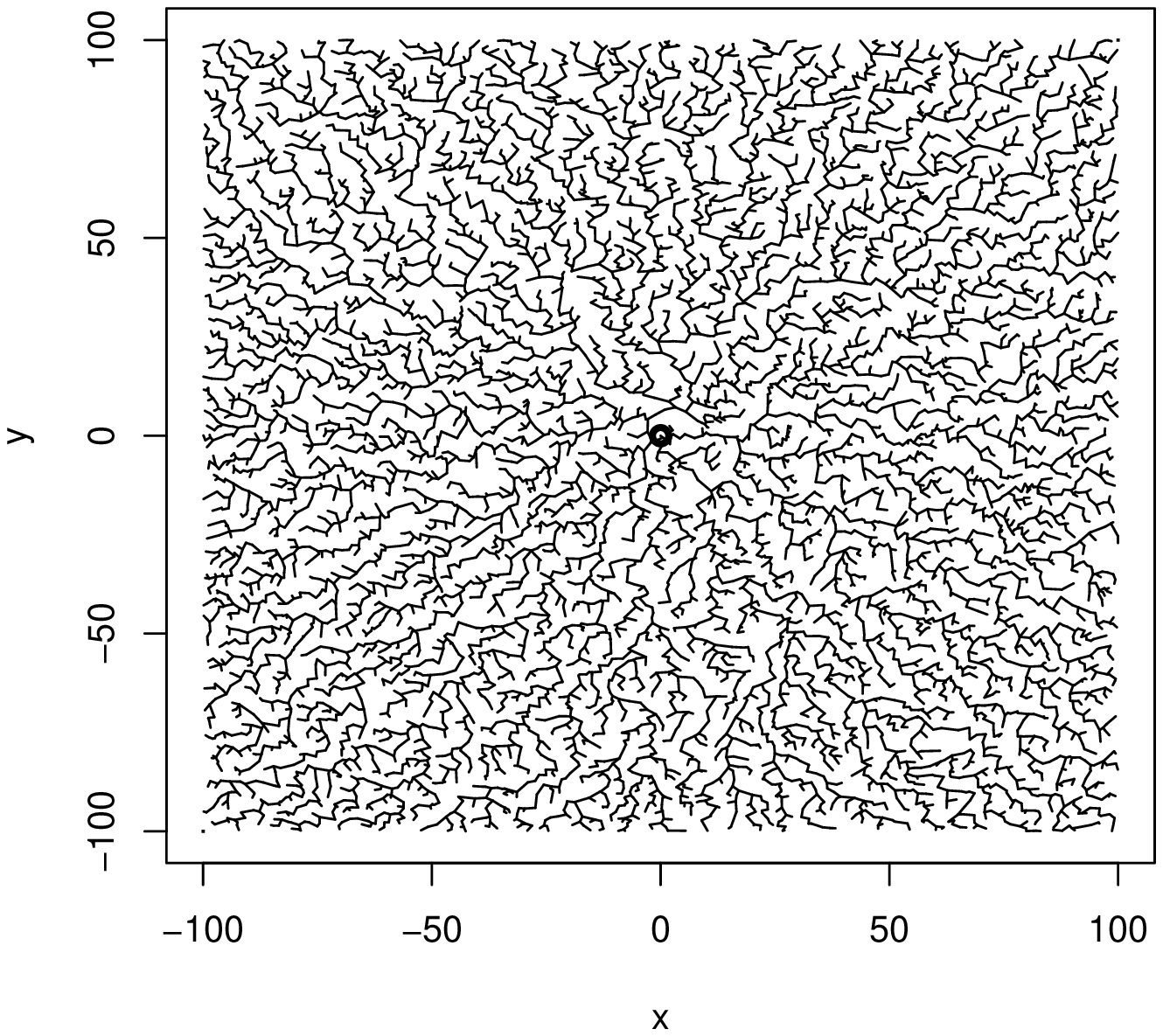} &
\includegraphics[width = 0.3\textwidth,trim=0cm 0cm 0cm 0cm]{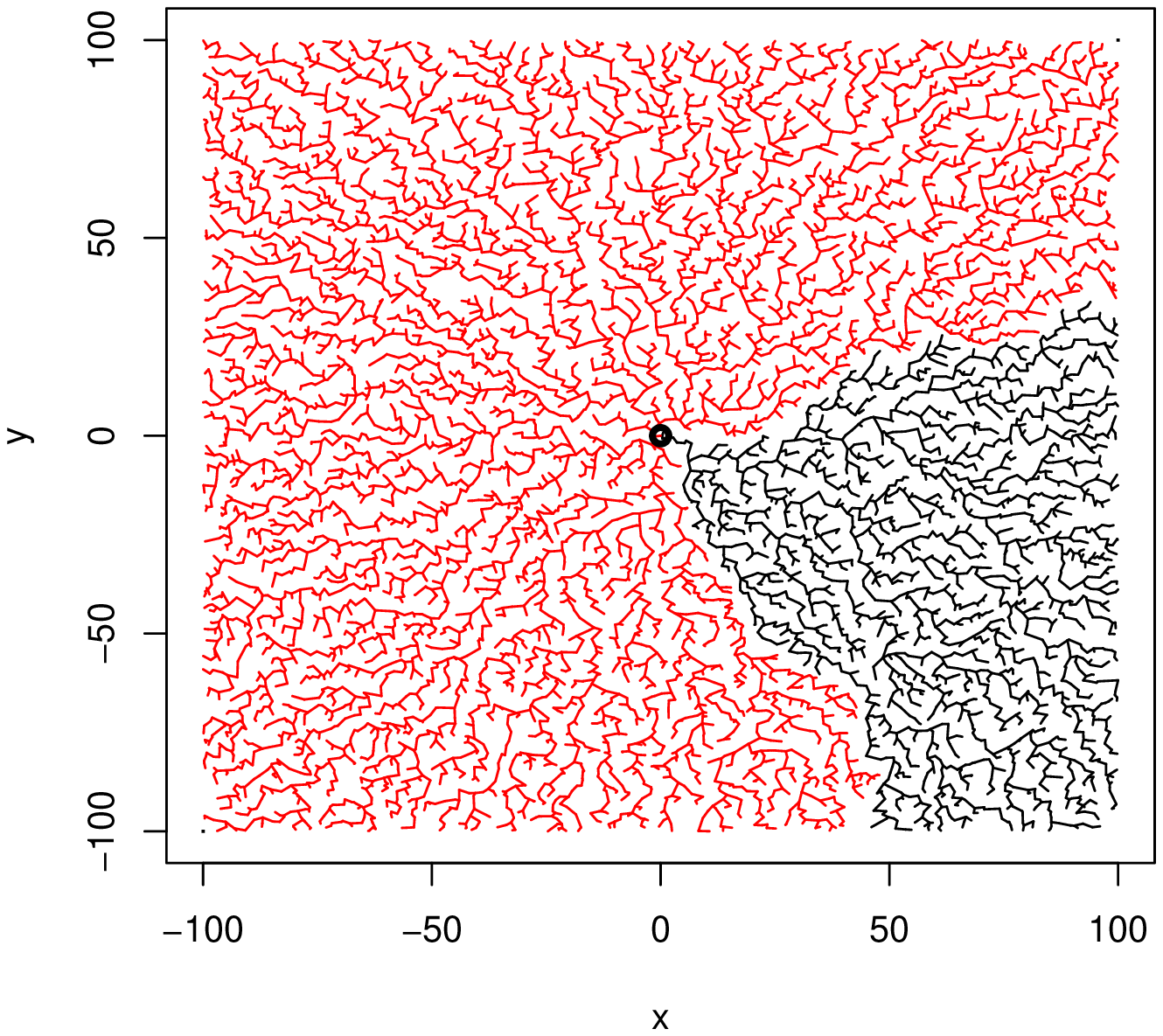}\\
(d) & (e) & (f) \\
\includegraphics[width = 0.3\textwidth,trim=0cm 0cm 0cm 0cm]{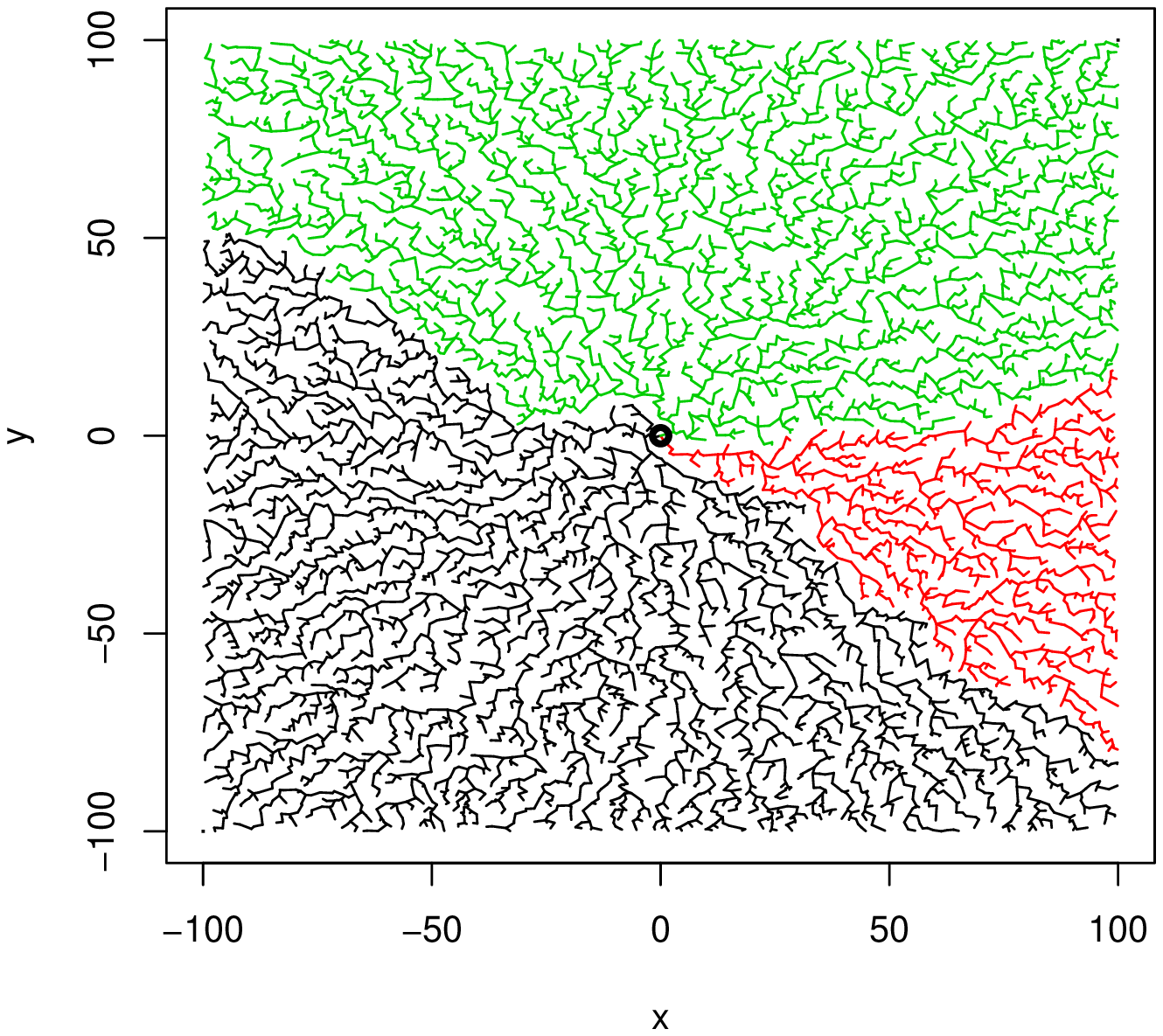} &
\includegraphics[width = 0.3\textwidth,trim=0cm 0cm 0cm 0cm]{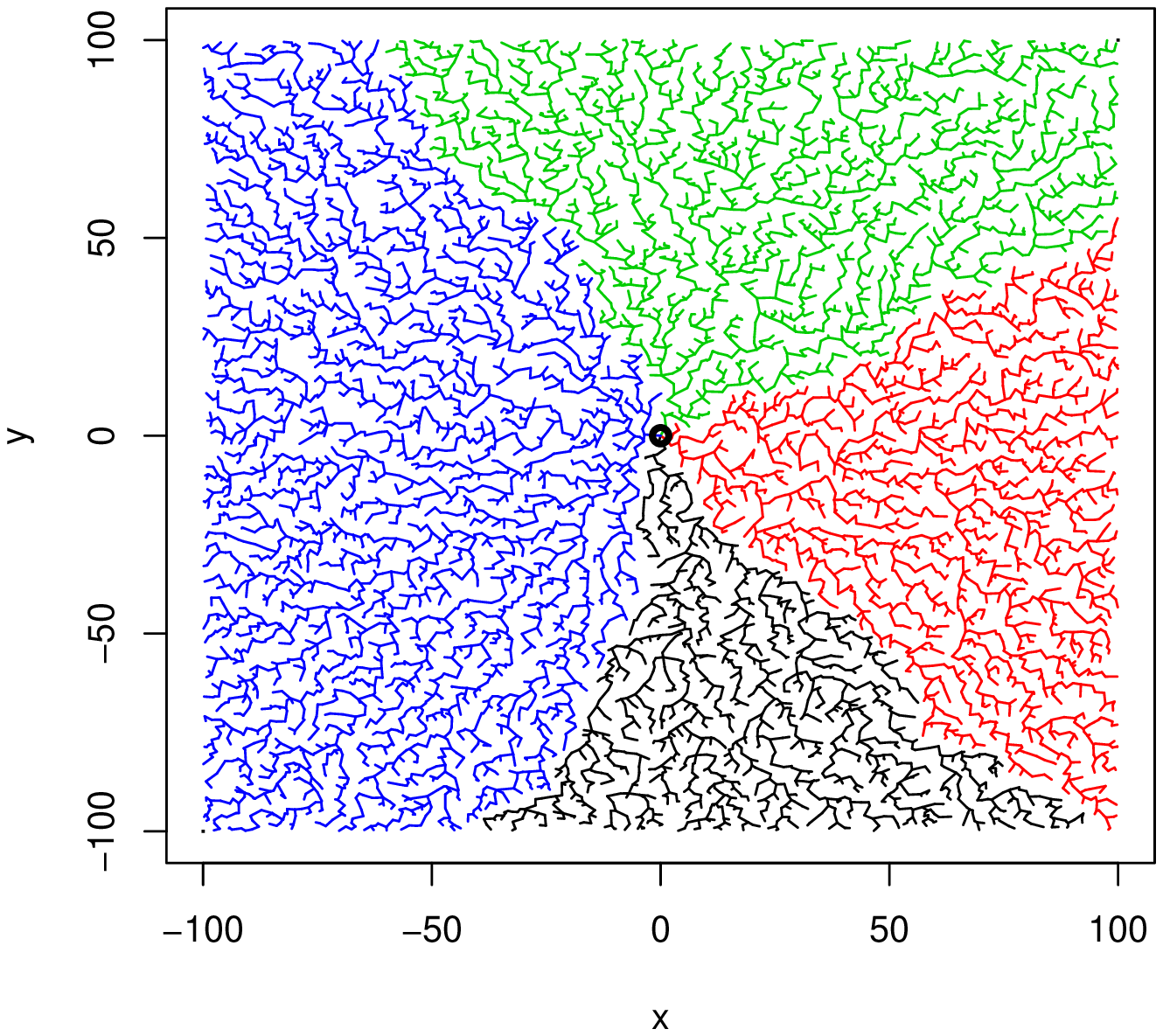} &
\includegraphics[width = 0.3\textwidth,trim=0cm 0cm 0cm 0cm]{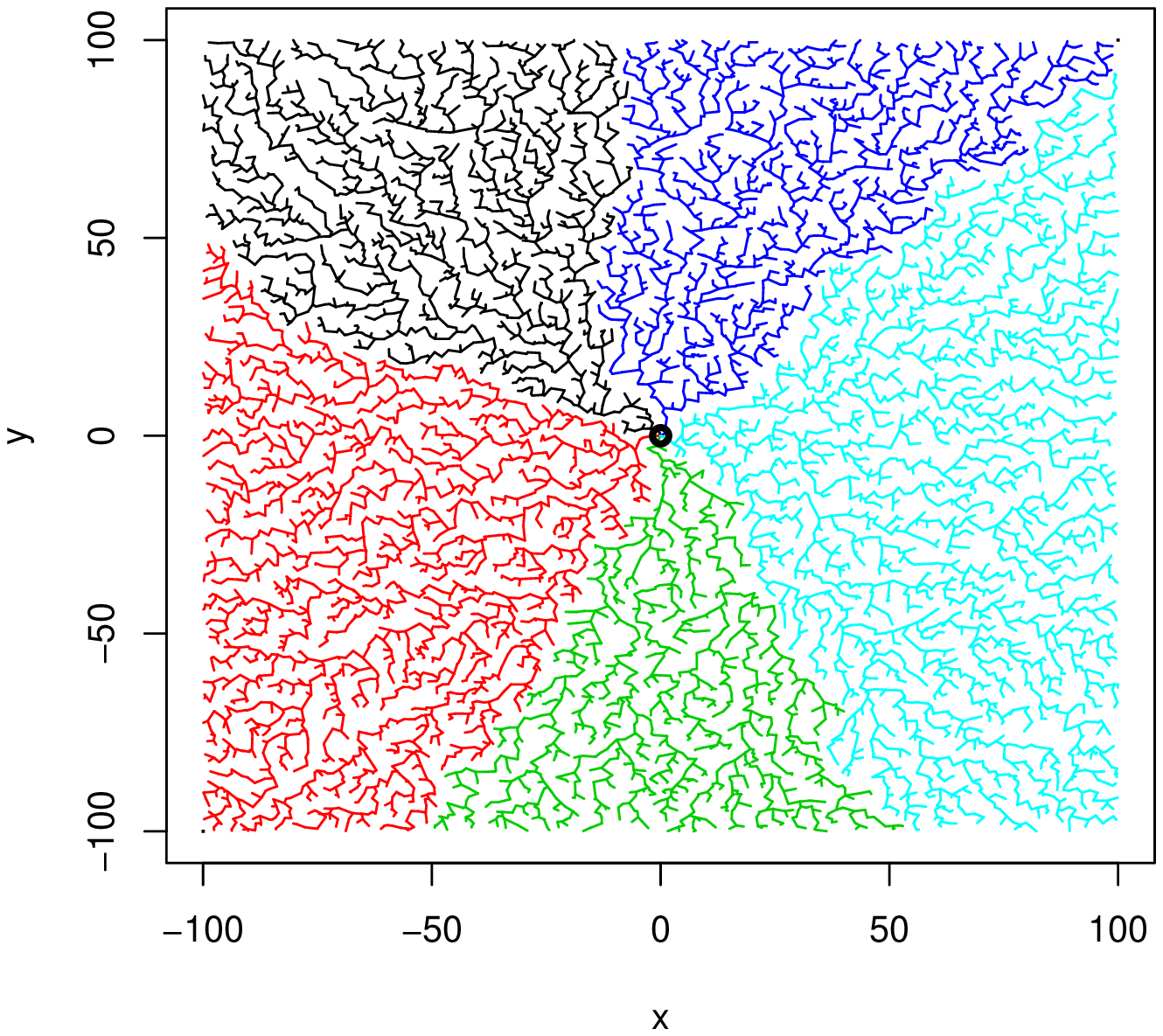}\\
\end{tabular}\end{center}\vspace{-0.5cm}
\caption{{\small \textit{(a) Empirical distribution for the number of children of $O$. Over 5000 simulations, 1 (resp. 2, 3, 4 and 5) child is obtained in 114 (resp. 2232, 2449,  203, 2) cases. Simulations of $m=1,\dots,5$ subtrees of the RST rooted at the children of $O$ are given from (b) to (f). They seem to be unbounded.}}}
\label{simu_fb}
\end{figure}

The next step is to define the \textit{competition interfaces}, i.e. the borders between the subtrees of the RST rooted at the children of the origin. To do so, let us introduce the spatially embedded version of the RST $\mathcal{T}$, denoted by $\mathbf{T}$, as the following subset of $\R^{2}$:
$$
\mathbf{T} = \bigcup_{X\in N\setminus \{O\}} [ X , \A(X) ] ~.
$$
For any positive real number $r$, the normalized trace of $\mathbf{T}$ over the sphere $S(O,r)$ is
$$
\mathbf{T}_{r} = \frac{1}{r} \left( \mathbf{T} \cap S(O,r) \right) ~.
$$
An element $u$ of $\mathbf{T}_{r}$ inherits its color from the element $ru\in\mathbf{T}$. So, for any given color $i$, we denote by $\mathbf{T}_{r}(i)$ the points of $\mathbf{T}_{r}$ with color $i$. By the noncrossing paths property of the RST, the points of $\mathbf{T}_{r}$ are ``gathered'' on the unit sphere $S(O,1)$ according to their color. This can be formalized as follows: for any $r>0$, $\theta_{1},\theta_{2},\theta_{3},\theta_{4}\in[0,2\pi)$ such that $(\theta_{1}-\theta_{3})(\theta_{2}-\theta_{3})>0$, $(\theta_{1}-\theta_{4})(\theta_{2}-\theta_{4})<0$,
$e^{\i\theta_{1}},e^{\i\theta_{2}}\in\mathbf{T}_{r}(i)$ and $e^{\i\theta_{3}},e^{\i\theta_{4}}\in\mathbf{T}_{r}$, at least one of the two points $e^{\i\theta_{3}}$ and $e^{\i\theta_{4}}$ is of color $i$.\\
For all $(\theta,\theta')\in[0,2\pi)^{2}$, let us denote by $a(\theta,\theta')$ (resp. $\overline{a}(\theta,\theta')$) the arc of the unit sphere from $e^{\i\theta}$ to $e^{\i\theta'}$ in the trigonometric sense, without (resp. with) the end points $e^{\i\theta}$ and $e^{\i\theta'}$. Furthermore, let $\T(i)$ be the subset of $\T$ with color $i$.

\begin{defn}[Competition interfaces]
\label{def:interface}
Given a couple of colors $(i,j)$ with $i\not=j$, there exists at most one couple $(\theta,\theta')\in[0,2\pi)^{2}$ such that
$$
e^{\i\theta} \in \mathbf{T}_{r}(i) , \; e^{\i\theta'} \in \mathbf{T}_{r}(j) \; \mbox{ and } \; a(\theta,\theta') \cap \mathbf{T}_{r} = \emptyset ~.
$$
When such a couple $(\theta,\theta')$ exists, we denote by $\theta_{r}(i,j)\in[0,2\pi)[$ the (direct) angle of the line coming from $O$ and bisecting the arc $a(\theta,\theta')$ in two equal parts. In this case, the competition interface between the sets $\T(i)$ and $\T(j)$ is defined as the curve:
$$
\varphi(i,j) = \{ r e^{\i\theta_r(i,j)}\in \C , \; \beta(i,j) < r < \partial(i,j) \} ~,
$$
where $\beta(i,j)$ and $\partial(i,j)$ are respectively defined as the infimum and the supremum of the set $\{r>0 , \theta_r(i,j) \mbox{ exists} \}$.
\end{defn}

\begin{figure}[!ht]
\begin{center}
\psfrag{i}{\small{$e^{\i\theta}$}}
\psfrag{j}{\small{$e^{\i\theta'}$}}
\psfrag{t}{\small{$\theta_{r}(i,j)$}}
\psfrag{d}{\small{$\Delta$}}
\includegraphics[width=6cm,height=6cm]{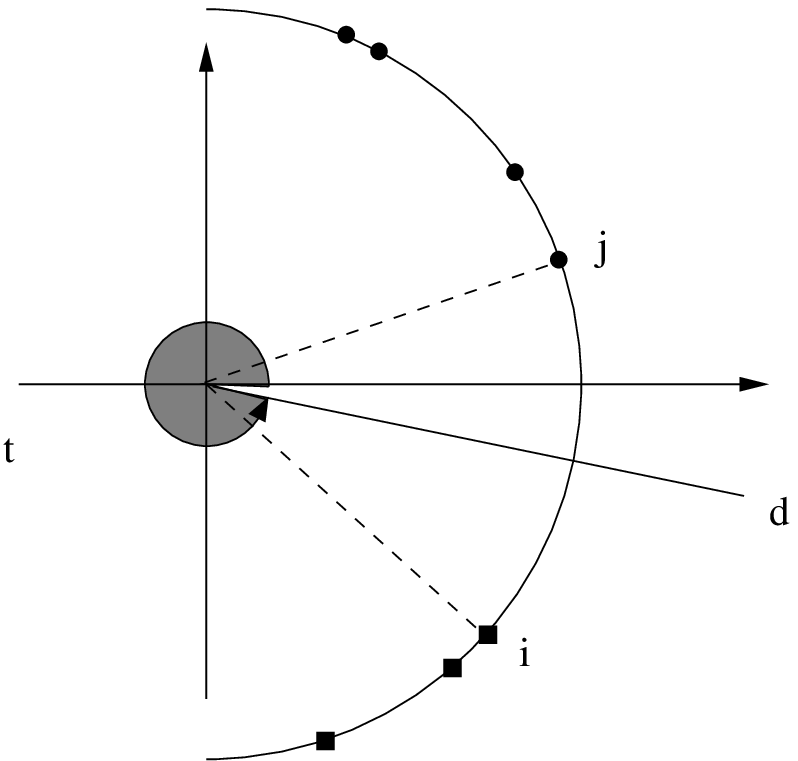}
\end{center}
\caption{\label{fig:angle_ij} {\small \textit{On the unit sphere, the black squares are points of $\mathbf{T}_{r}(i)$ while black circles are points of $\mathbf{T}_{r}(j)$. The arc $a(\theta,\theta')$ is divided in two equal parts by the line $\Delta$ whose angle (represented in grey) is $\theta_{r}(i,j)$.}}}
\end{figure}

From $\beta(i,j)$ to $\partial(i,j)$, the trees $\T(i)$ and $\T(j)$ evolve in the plane side by side, separated by the competition interface $\varphi(i,j)$. The real numbers $\beta(i,j)$ and $\partial(i,j)$ can respectively be interpreted as the birth and death times of the competition interface $\varphi(i,j)$. When $\partial(i,j)=+\infty$, both sets $\T(i)$ and $\T(j)$ are unbounded. When $\partial(i,j)<+\infty$, one of the two sets $\T(i)$ and $\T(j)$ is included in the closed ball $\overline{B}(O,\partial(i,j))$, say $\T(j)$. In this case, $\partial(i,j)$ coincides with another death time $\partial(j,k)$ and two situations may occur according to the color $k$. Either $k=i$ which means $i$ is the only existing color outside the ball $\overline{B}(O,\partial(i,j))$ and there is no competition interface beyond that ball. Or $k$ is a third color (different from $i$ and $j$). Then, the competition interface $\varphi(i,k)$ extends $\varphi(i,j)$ and $\varphi(j,k)$ (until its de!
 ath time $\partial(k,j)$). Its birth time satisfies:
$$
\beta(i,k) = \partial(i,j) = \partial(j,k) > 0 ~.
$$
Let us remark that the application $r\mapsto\theta_{r}(i,j)$ may be discontinuous. Finally, notice that $\theta_r(i,j)\not= \theta_r(j,i)$ and that one may exist and the other not. So, we distinguish the interfaces $\varphi(i,j)$ and $\varphi(j,i)$.\\

Our first result states there can be up to five unbounded competition interfaces with positive probability.

\begin{thm}
\label{arbresinfinis}
For any $m\in\{1,2,3,4,5\}$, there exist (exactly) $m$ unbounded subtrees of $\mathcal{T}$ with different colors, with positive probability. In other words, for any $m\in\{0,2,3,4,5\}$, there exist (exactly) $m$ unbounded competition interfaces, with positive probability.
\end{thm}

\begin{table}[ht]
\begin{center}
\begin{tabular}[p]{cc}
(a) & (b) \\
\includegraphics[width = 0.3\textwidth,trim=0cm 0cm 0cm 0cm]{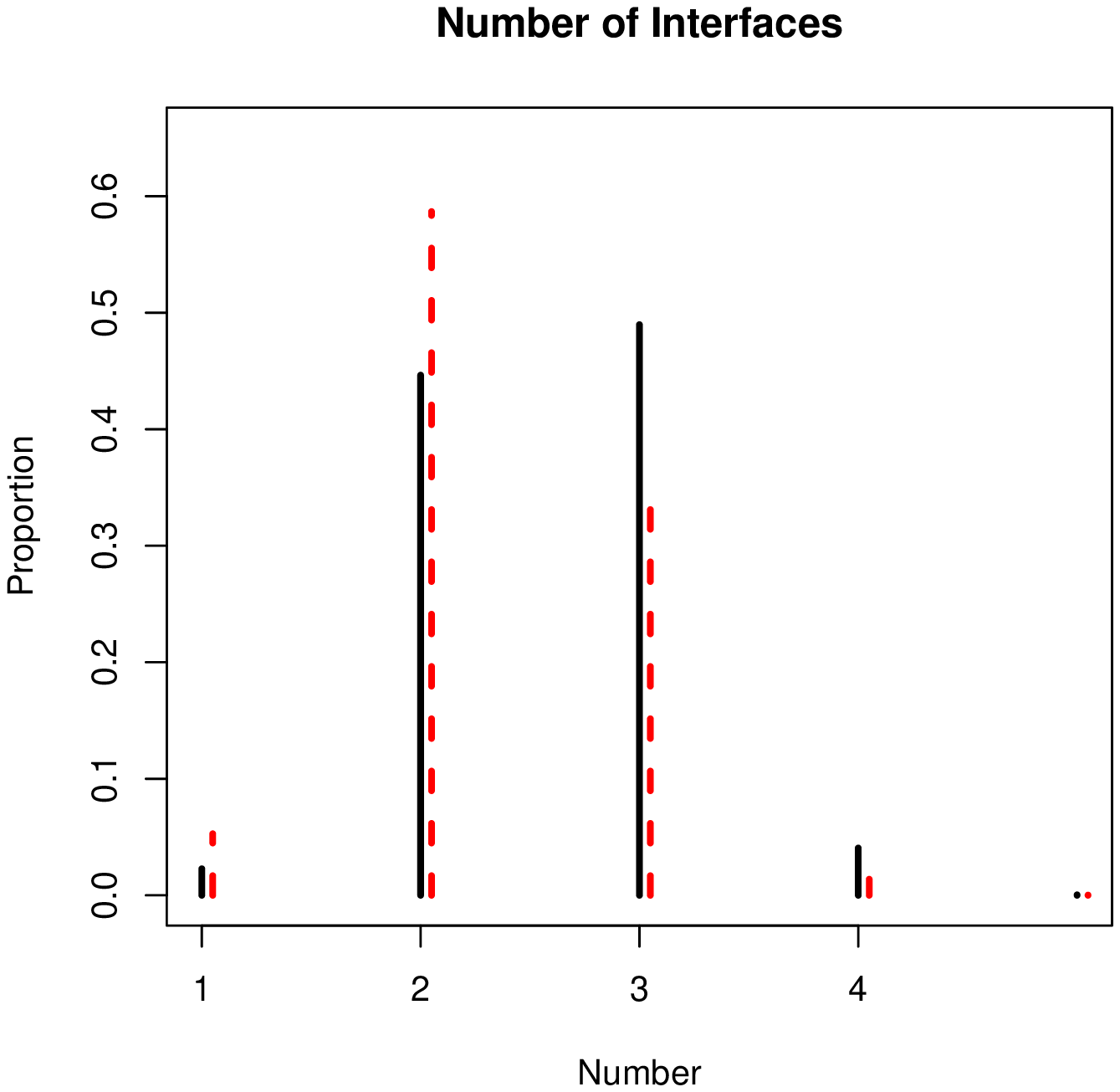}
&
\vtop{\vspace{-3cm} \hbox{
\begin{tabular}{|c|ccccc|}
\hline
m & 1 & 2 & 3 & 4 & 5 \\
\hline
Children of $O$ & 2.28 & 44.64 & 48.98&  4.06 & 0.04 \\
Unbounded subtrees & 5.28& 58.68 & 34.64 & 1.38 & 0.02 \\
\hline
\end{tabular}}}
\end{tabular}
\end{center}\vspace{-0.5cm}
\caption{{\small \textit{(a) Empirical distributions, obtained after $N=5000$ simulations, for the number of children of $O$ (in black plain lines) and for the number of unbounded subtrees (in red dotted lines). Percentages are given in the table (b). The two distributions are different since the tree associated with each child of $O$ is not necessarily unbounded. About its second line, let us point out the cases $m\in\{4,5\}$ are very rare (less than $2\%$ of the simulations) compared with the cases $m\in\{2,3\}$ (more than $93\%$). Actually, configurations corresponding to $m\in\{4,5\}$ are very constrained around the origin, therefore rare.}}}
\label{fig3}
\end{table}

Our proof relies on Part $(i)$ of Theorem \ref{HN1}. Thinning and local modification of the PPP are other ingredients.

\begin{proof}
We consider the cases $m\in \{1,\dots,5\}$ separately.\\
\fbox{$\mathbf{m=5}$} Our purpose is to construct a set of configurations of $N$, with a positive probability, on which there are five children of the origin $O$ giving birth to infinite subtrees.\\
For any $1\leq k\leq 5$, Part $(i)$ of Theorem \ref{HN1} ensures the existence a.s. of a semi-infinite path $\gamma_{k}$ with asymptotic direction $2k\pi/5$. Hence, for $\varepsilon>0$ and with probability $1$, there exists a (random) radius $r_{k}$ such that $\gamma_{k}$ is included in the cone section
$$
C_{2k\pi/5,\varepsilon,r_{k}} = \left\{ \rho e^{\i\theta} \; ; \; \rho > r_{k} \; \mbox{ and } \; \left| \theta- 2k\pi/5 \right| < \varepsilon \right\}
$$
for any integer $1\leq k\leq 5$. Without loss of generality, we can require that $\gamma_{k}$ starts from a vertex $X_{k}\in N$ whose norm satisfies $r_{k}<|X_{k}|\leq r_{k}+1$ and for $r_{k}$ to be a positive integer. Hence, writing
$$
A_{\varepsilon}(r_{1},\ldots,r_{5}) = \left\{ \begin{array}{c}
\mbox{ for any } 1\leq k\leq 5 , \mbox{ there exists a semi-infinite path } \gamma_{k} \\
\mbox{ included in the cone } C_{2k\pi/5,\varepsilon,r_{k}} \mbox{ and starting from } \\
\mbox{ a vertex } X_{k} \mbox{ satisfying } r_{k}<|X_{k}|\leq r_{k}+1
\end{array} \right\} ~,
$$
we get that for all $\varepsilon>0$, there exist some (deterministic) radii $r_{1},\ldots,r_{5}\in\N^*$ such that $A_{\varepsilon}(r_{1},\ldots,r_{5})$ occurs with positive probability.\\
Let $R=\max\{r_{k}+1 ; 1\leq k\leq 5\}$ and $V_{\varepsilon}(r_{1},\ldots,r_{5})$ be the complementary set of the five cones in the ball $B(O,R)$:
$$
V_{\varepsilon}(r_{1},\ldots,r_{5}) = B(O,R) \setminus \Big[\Big( \cup_{k=1}^{5} C_{2k\pi/5,\varepsilon,r_{k}} \Big) \cup \{O\} \Big] ~.
$$
Now, we are going to change the configuration of the PPP $N$ in $V_{\varepsilon}(r_{1},\ldots,r_{5})$ in such a way that the $X_k$'s are all of different colors. Let $\widetilde{N}=N\cap V_{\varepsilon}^{c}(r_{1},\ldots,r_{5})$ be the thinned PPP obtained by deleting all the points of $N$ belonging to $V_{\varepsilon}(r_{1},\ldots,r_{5})$ (\eg Jacod and Shiryaev \cite{jacod}, II.4.b). It is crucial to remark that deleting the points of $V_{\varepsilon}(r_{1},\ldots,r_{5})$ does not affect the occurrence of $A_{\varepsilon}(r_{1},\ldots,r_{5})$. In other words, if $N$ satisfies the event $A_{\varepsilon}(r_{1},\ldots,r_{5})$, so does $\widetilde{N}$;
$$
\P \left( \widetilde{N} \in A_{\varepsilon}(r_{1},\ldots,r_{5}) \right) \geq \P \left( N \in A_{\varepsilon}(r_{1},\ldots,r_{5}) \right) > 0 ~.
$$
Now, let us consider a PPP $\hat{N}$ on $V_{\varepsilon}(r_{1},\ldots,r_{5})$ with intensity $1$. Let us denote by $r$ the minimum of the $r_k$'s.\\
The event $\hat{N}\in B_{\varepsilon}(r_{1},\ldots,r_{5})$ is defined by the three following conditions.
\begin{itemize}
\item[$(\clubsuit)$] For any $k\in\{1,\dots,5\}$, if $r_{k}>r$ then for all integers $r \leq n \leq r_{k}-1$,
$$
\hat{N} \left( B(n e^{\i 2k\pi/5} , \varepsilon) \right) = 1 ~,
$$
else
$$
\hat{N} \left( B(r e^{\i 2k\pi/5} , \varepsilon) \cap B(O,r)  \right) = 1 ~.
$$
\item[$(\diamondsuit)$] For any $k\in \{1,\ldots,5\}$ and for all integers $n$ such that $0\leq n\leq (R-r_k-1)/2\varepsilon$:
$$
\hat{N}\left(B\big((r_k+1+2n\varepsilon) e^{\i (2k\pi/5\pm2\varepsilon)},\varepsilon\big) \cap V_{\varepsilon}(r_{1},\ldots,r_{5}) \right)=1 ~.
$$
\item[$(\heartsuit)$] The previous points are the only ones of $\hat{N}$.
\end{itemize}
It is clear that the event $\hat{N}\in B_{\varepsilon}(r_{1},\ldots,r_{5})$ occurs with positive probability, for all $\varepsilon>0$. Roughly speaking, the points of $\hat{N}$ introduced in $(\clubsuit)$ form a chain from $re^{\i 2k\pi/5}$ to $(r_{k}-1)e^{\i 2k\pi/5}$, for any index $k$ such that $r_{k}>r$. See Figure \ref{fig:5arbres}.

\begin{figure}[!ht]
\begin{center}
\includegraphics[width=11cm,height=11cm]{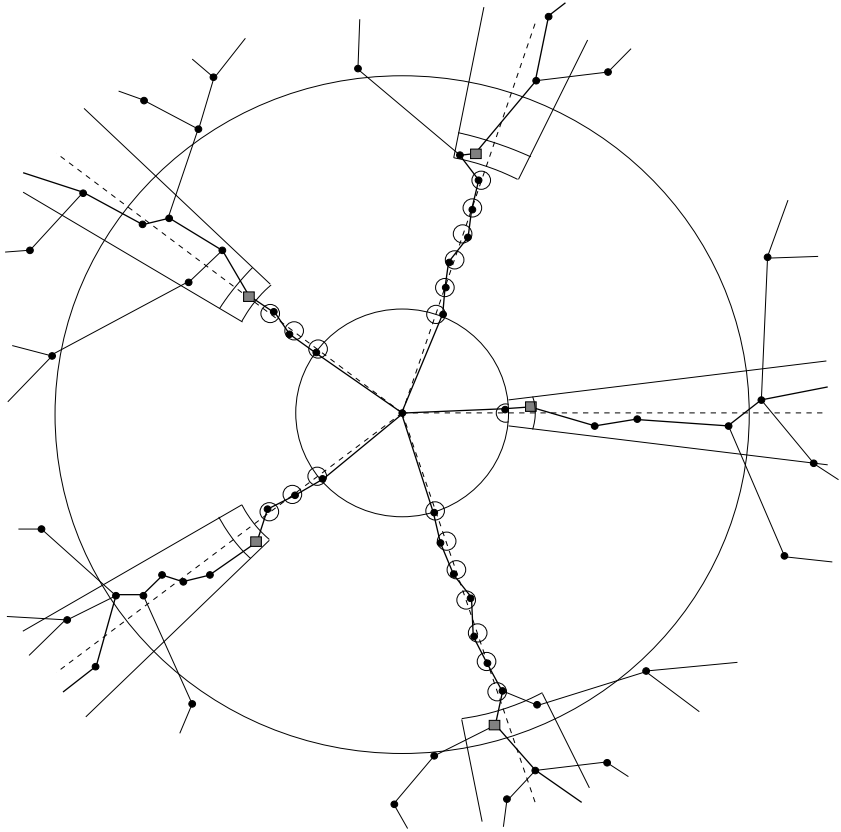}
\end{center}
\caption{\label{fig:5arbres} {\small \textit{RST of the PPP $N$ satisfying both events $A_{\varepsilon}(r_{1},\ldots,r_{5})$ and $B_{\varepsilon}(r_{1},\ldots,r_{5})$. Beware the fact that, in order not to overload the figure, the condition $(\diamondsuit)$ of $B_{\varepsilon}(r_{1},\ldots,r_{5})$ has not been represented. The two balls are centered at $O$ with radii $r=\min_{k\in \{1,\dots,5\}} r_{k}$ and $R=\max_{k\in \{1,\dots,5\}} r_{k}+1$. The $X_{k}$'s are represented by big gray squares while the other points of $N$ by small black circles.}}}
\end{figure}

\noindent
On Figure \ref{fig:5arbres}, imagine that $R=r_{4}+1$ is much larger than $r=r_{5}$ (indeed, we have no control on the $r_{k}$'s). Henceforth, the semi-infinite path $\gamma_{5}$ could prefer to branch on the points of $\hat{N}$ introduced in $(\clubsuit)$ and with direction $8\pi/5$ rather than on $X_{5}$. To prevent this situation from occurring, we contain each path $\gamma_k$ in the cone $C_{2k\pi/5,3\varepsilon,r_{k}}$ thanks to the points of $\hat{N}$ introduced in $(\diamondsuit)$. These points form ``landing runways'' for the $\gamma_{k}$'s (they may also change slightly the $\gamma_{k}$'s).\\
Let us denote by $A_{\varepsilon}$ and $B_{\varepsilon}$ the events $A_{\varepsilon}(r_{1},\ldots,r_{5})$ and $B_{\varepsilon}(r_{1},\ldots,r_{5})$. Then,
$$
\{ \widetilde{N} \in A_{\varepsilon} \} \cap \{ \hat{N} \in B_{\varepsilon} \} \subset \{ \widetilde{N}+\hat{N} \in A_{3\varepsilon} \cap B_{\varepsilon} \} ~,
$$
where $\widetilde{N}+\hat{N}$ denotes the superposition of the two processes $\hat{N}$ and $\widetilde{N}$. These two processes can also be assumed independent. In this case, $\widetilde{N}+\hat{N}$ is still a PPP on $\mathbb{R}^{2}$. It follows:
\begin{eqnarray*}
\P ( N \in A_{3\varepsilon} \cap B_{\varepsilon} ) & = & \P ( \widetilde{N}+\hat{N} \in A_{3\varepsilon} \cap B_{\varepsilon} ) \\
& \geq & \P ( \widetilde{N} \in A_{\varepsilon} \; , \hat{N} \in B_{\varepsilon} ) \\
& \geq & \P ( \widetilde{N} \in A_{\varepsilon} ) \P ( \hat{N} \in B_{\varepsilon} ) \; > \; 0 ~.
\end{eqnarray*}
To conclude the proof, it remains to prove that the above event implies the existence of (at least) five unbounded subtrees of $\mathcal{T}$ with different colors. Actually, there will be exactly five ones since the degree of $O$ is a.s. upperbounded by $5$. Let us denote by $Y_{k}$ the point of $N$ belonging to the ball $B(r e^{\i 2k\pi/5},\varepsilon)$. On the event $N\in A_{3\varepsilon}\cap B_{\varepsilon}$, the point $X_{k}$ is a descendant of $Y_{k}$ for any $k$. Hence, the subtrees rooted at $Y_{1},\ldots,Y_{5}$ are unbounded. Finally, it suffices to remark the $Y_{k}$'s have $O$ as common ancestor. Indeed, each $Y_{k}$ is at distance from $e^{\i 2k\pi/5}$ smaller than $\varepsilon$. So,
\begin{eqnarray*}
| Y_{k+1} - Y_{k} | & \geq & | re^{\i 2(k+1)\pi/5} - re^{\i 2k\pi/5} | - 2 \varepsilon \\
& \geq & 2 r \sin(\pi/5)-2\varepsilon\\ 
& \geq & 1.17 r - 2 \varepsilon ~,
\end{eqnarray*}
which is larger than the maximal distance between $Y_{k}$ and $O$, i.e. $r+\varepsilon$, for $\varepsilon$ small enough (using $r\in\N^*$).

\medskip
\noindent
\fbox{$\mathbf{m\in\{3,4\}}$} The previous construction applied to $m\in\{3,4\}$ allows us to state that with positive probability, the origin $O$ has at least $m$ descendants from which $m$ unbounded trees arise. Now, so as to ensure the number of unbounded subtrees of different colors is exactly $m$, an additional precaution must be taken. Precisely, a fourth condition is added to the event $\hat{N}\in B_{\varepsilon}$:
\begin{itemize}
\item[$(\spadesuit)$] For any $k\in\{1,\dots,m\}$, the argument of the point $Y_{k}$ of $\hat{N}\cap B(r e^{\i 2k\pi/m},\varepsilon)$ belongs to $(2k\pi/m-\varepsilon,2k\pi/m)$.
\end{itemize}
Thanks to $(\spadesuit)$, each sector of the ball $B(O,r+1)$ with angle $2\pi/m$ contains (at least) one of the points $Y_1,\dots Y_m$. Assume $N\in A_{3\varepsilon}\cap B_{\varepsilon}$ which still occurs with positive probability. By construction, the origin $O$ has exactly $m$ children in the ball $B(O,R)$. Let us consider a point $X\in N\setminus \{O,Y_1,\dots,Y_m\}$ such that $|X|\geq R\geq r+1$. Then $B(O,|X|)\cap B(X,|X|)$ contains a sector of the ball $B(O,r+1)$ with angle $2\pi/3$ and so one of the $Y_1,\dots Y_m$. The origin $O$ cannot be the ancestor of $X$. This proves that $O$ is exactly of degree $m$ and ends the proof.\\
This latter argument no longer works when $m$ is equal to $1$ or $2$.

\medskip
\noindent
\fbox{$\mathbf{m=2}$} Following the construction for $m=5$, there exists $r_1 $ and $r_2>0$ such that there exist with positive probability two semi-infinite paths $\gamma_1$ and $\gamma_2$ included in the cones $C_{0,\varepsilon,r_1}$ and $C_{\pi,\varepsilon,r_2}$. The following event has a positive probability:
\begin{itemize}
 \item For a given increasing subsequence $(\theta_j)_{j\in \N}$ of $[0,\pi)$ with a sufficiently small step, and for a sufficiently small $\varepsilon>0$:
$$N\Big(B\big((r_1\wedge r_2)(1+\cos(\theta_j))e^{\i \theta_j},\varepsilon\big)\Big)=1,\qquad N\Big(B\big(-(r_1\wedge r_2)(1+\cos(\theta_j))e^{\i \theta_j},\varepsilon\big)\Big)=1,$$
\item For all integers $n$ and $m$ such that $0\leq n\leq (r_2 -r_1)/2\varepsilon$ and $0\leq n\leq (r_1 -r_2)/2\varepsilon$, if they exist:
$$N\Big(B\big((r_1+2n\varepsilon,0),\varepsilon\big)\Big)=1,\qquad N\Big(B\big((0,r_2-2n\varepsilon),\varepsilon\big)\Big)=1.$$
\item The rest of $B(O,r_1\vee r_2)$ is empty.
\end{itemize}
The idea is that in $B(O,r_1\wedge r_2)$, the points are roughly aligned following the reunion of two cardiods $\{\rho(\theta)=\pm (1+\cos(\theta)),\, \theta\in [0,\pi)\}$. Notice that this curve is differentiable at $O$ with  a horizontal tangent. If $r_1<r_2$, we add points along the line segment $[(r_1,0),(r_2,0)]$. If the $\theta_j$'s define a sufficiently fine subdivision of $[0,\pi)$, then there cannot be more than two descendants of $O$ by construction. We conclude as in the case $m=5$.\\
We remark that the two semi-infinite paths previously built have asymptotic directions opposed to the argument of the descendant of $O$ from which they stem.

\medskip
\noindent
\fbox{$\mathbf{m=1}$} Since the RST $\mathcal{T}$ is unbounded, it suffices to prove that the origin $O$ may have only one child with positive probability.\\
From $z_{1}=e^{\i\pi/3}$, we build five complex numbers $z_{2},\ldots,z_{6}$ by the following induction: for $k\geq 2$, $z_{k}=|z_{k}| e^{\i k\pi/3}$ whose modulus $|z_{k}|$ is such that $|z_{k}-z_{k-1}|<|z_{k}|$. This construction forces $|z_{k}|>|z_{k-1}|$. Let $\varepsilon>0$ small enough such that $|z_{k}|-\varepsilon>|z_{k-1}|+\varepsilon$. Hence, the six balls $B(z_{1},\varepsilon),\ldots,B(z_{6},\varepsilon)$ do not overlap. Let $\Omega_{\varepsilon}$ be the event
$$
\forall 1\leq k \leq 6 , \; N(B(z_{k},\varepsilon)) = 1 \; \mbox{ and } \; N(B(O,|z_{6}|+\varepsilon)) = 7
$$
(these 7 points including the origin). For all $\varepsilon>0$, $\P(\Omega_{\varepsilon})>0$. So, it remains to choose $\varepsilon>0$ small enough in order to ensure that, on the event $\Omega_{\varepsilon}$, the origin $O$ has only one child.\\
Let us denote by $X_{k}$ the point of $N\cap B(z_{k},\varepsilon)$. Since $N\cap B(O,|X_{1}|)$ is reduced to $O$, the ancestor of $X_{1}$ is the origin $O$. Thus, for $2\leq k\leq 6$, we can choose $\varepsilon$ such that
$$
|X_{k}-X_{k-1}| \leq |z_{k}-z_{k-1}| + 2\varepsilon < |z_{k}-O| - \varepsilon \leq |X_{k}-O| ~.
$$
This condition does not prove that $X_{k-1}$ is the ancestor of $X_{k}$, but it is not $O$. Finally, let $X$ be a point of the PPP $N$ which does not belong to $B(O,|z_{6}|+\varepsilon)$. The set $B(X,|X|)\cap B(O,|X|)$ contains an angular sector of the ball $B(O,|z_{6}|+\varepsilon)$ with central angle $2\pi/3$. So, it also contains one of the $X_{k}$'s, preventing $X$ from being a child of $O$. To sum up, $X_{1}$ is the only child of the origin $O$.
\hfill $\Box$ \end{proof}

Let $\Omega(i,j)$ be the event corresponding to an unbounded competition interface $\varphi(i,j)$. It occurs with a positive probability thanks to Theorem \ref{arbresinfinis}. Recall that $\varphi(i,j)$ separates the two colored subtrees $\T(i)$ and $\T(j)$ according to the trigonometric sense.\\
The next result states that $\varphi(i,j)$ has a.s. an asymptotic direction on the event $\Omega(i,j)$. In other words, if $\T(i)$ is unbounded then it asymptotically behaves as a cone.

\begin{prop}
\label{prop:cvps}
On the event $\Omega(i,j)$, the sequence $(\theta_{r}(i,j))_{r>\beta(i,j)}$ converges a.s. to a random angle $\theta(i,j)\in[0,2\pi)$.
\end{prop}

\begin{proof}
Let us consider the event $\Omega(i,j)$ satisfied. Let $X(i)$ and $X(j)$ be the children of the origin of color $i$ and $j$. On $\Omega(i,j)$, both subtrees $\mathcal{T}_{X(i)}$ and $\mathcal{T}_{X(j)}$ are unbounded. Recall that $\overline{\gamma}_{X(i)}$ denotes the highest (in the trigonometric sense) semi-infinite path in $\mathcal{T}_{X(i)}$ (see the proof of Proposition \ref{prop:<2}). In the same way, $\underline{\gamma}_{X(j)}$ is the lowest one in $\mathcal{T}_{X(j)}$. On $\Omega(i,j)$, the region delimited by $\overline{\gamma}_{X(i)}$ and $\underline{\gamma}_{X(j)}$ (in the trigonometric sense) only contains finite paths. It may also contain some vertices of a third color (different from $i$ and $j$). Then, by Parts $(i)$ and $(ii)$ of Theorem \ref{HN1}, $\overline{\gamma}_{X(i)}$ and $\underline{\gamma}_{X(j)}$ have the same asymptotic direction, say $\theta(i,j)$. To conclude it suffices to remark that the competition interface $\varphi(i,j)$ is trapped between !
 $\overline{\gamma}_{X(i)}$ and $\underline{\gamma}_{X(j)}$. It then admits the same direction.
\hfill $\Box$ \end{proof}

Proposition \ref{prop:cvps} says that every competition interface that separates the colors $i$ and $j$ has an asymptotic direction $\theta(i,j)$. The next proposition states a result on the distribution of the asymptotic directions, which remains however partial. Recall that we use the labels $\underline{1},\dots \underline{i}\dots$ when the subtrees rooted at $O$ are labeled randomly. If the marginal distributions of the $\theta(\underline{i},\underline{i+1})$'s are easy to obtain, it is not the case for the distributions of the $\theta(i,i+1)$'s which necessitate the knowledge of the joint distributions of the asymptotic directions (or equivalently, the distribution of the sectors between the competition interfaces). Section \ref{section:conjectures} provides numerical simulations and conjectures.

\begin{prop}\label{prop_unif}
Conditionally on having $m$ infinite trees, and when the tree with color 1 is drawn randomly, the asymptotic directions $\theta(\underline{i},\underline{i+1})$ are uniformly distributed on $[0,2\pi)$.\\
Moreover the distribution of $\theta(i,j)$, on $\Omega(i,j)$, admits a density with respect to the Lebesgue measure on $[0,2\pi)$.
\end{prop}

\begin{proof}
The first part results from invariance by translation and from the characterization of the Lebesgue measure on the circle as the unique measure invariant by any rotation.\\
The event $\{\theta(i,j)=\alpha\}$ implies the existence of at least two semi-infinite paths with the deterministic direction $\alpha$. This is forbidden by Proposition \ref{prop:<2}. So, $\theta(i,j)$ has no atom (when it exists). In fact, the distribution of $\theta(i,j)$ is even absolutely continuous with respect to the Lebesgue measure $\lambda$ on $[0,2\pi)$. Let $A$ be a measurable subset of $[0,2\pi)$ such that $\lambda(A)=0$. Let us denote by $M$ the random number of interfaces that exist. Then for $i\not=j$, since $\theta(i,j)$ corresponds to one of the $\theta(\underline{k},\underline{l})$ when we relabel the subtrees rooted at $O$ randomly:
\begin{multline*}
\P\big(\{\theta(i,j)\in A\}\cap \{M\geq 2\} \cap \Omega(i,j)\big)=  \sum_{m=2}^5 \P\big(\{\theta(i,j)\in A\}\cap \{M=m\}\cap \Omega(i,j)\big)\\
\begin{aligned}
\leq &  \P\big(\bigcup_{i\not= j\in \{1,\dots,m\}}\{\theta(\underline{i},\underline{j})\in A\}\cap \{M=m\}\cap \Omega(\underline{i},\underline{j})\big)\\
\leq  & \sum_{m=2}^5 \sum_{i\not= j\in \{1,\dots,m\}}\P\big(\{\theta(\underline{i},\underline{j})\in A\}\cap \{M=m\}\big)\leq 0,
\end{aligned}
\end{multline*}since $\lambda(A)=0$. Radon-Nikodym's theorem concludes the proof.
\hfill $\Box$ \end{proof}

We conclude this section by a corollary that states that the asymptotic directions of competition interfaces and of semi-infinite paths are related.
 \begin{cor}
 The asymptotic direction of the competition interface $\varphi(i,j)$ belongs to the (random) set $D$ of directions with at least two semi-infinite paths. This set is a.s. dense in $[0,2\pi)$ and countable.
 \end{cor}

 \begin{proof}The fact that $D$ is dense in $[0,2\pi)$ follows from Part $(iii)$ of Theorem \ref{HN1}. It is also a.s. countable. Indeed, let us consider the set $\Gamma$ of couples $(\gamma_{1},\gamma_{2})$ of different semi-infinite paths of the RST such that the region they delimit (in the trigonometric sense) contains only finite paths. Associating to each element $(\gamma_{1},\gamma_{2})$ of $\Gamma$ the child in $\gamma_{1}$ of their bifurcation point, we get an injective function from $\Gamma$ to the PPP $N$. Consequently, $\Gamma$ is a.s. countable. Moreover, Parts $(i)$ and $(ii)$ of Theorem \ref{HN1} allow to associate to each element $(\gamma_{1},\gamma_{2})$ of $\Gamma$ their common asymptotic direction. This provides a surjective function from $\Gamma$ onto the set $D$. Hence $D$ is a.s. countable.
\hfill $\Box$ \end{proof}

\section{Distribution of the $\theta(i,j)$'s and conjectures}\label{section:conjectures}

In this section, we provide some clues and conjectures that may help understanding the distribution of the vector $(\theta(1,2),\dots,\theta(m-1,m),\theta(m,1))$ of asymptotic directions of the interfaces, given that there are $m$ unbounded trees and assuming that the latter are labeled by following the trigonometric sense.


For this purpose, it is equivalent to study the distribution of the sectors $(\phi(i+1):=\theta(i+1,i+2)-\theta(i,i+1),\, i\in \{1,\dots,m\})$ (with the convention that $\theta(m,m+1)=\theta(m,1)$ and $\theta(m+1,m+2)=\theta(1,2)$), which characterize the asymptotic width of the unbounded trees.
\begin{prop}\label{lemm:conj2}
Conditionally on having $m$ unbounded trees, the angles between two interfaces are identically distributed with expectation $2\pi/m$.
\end{prop}
Notice first that this rules out the possibility that the asymptotic directions $\theta(i,j)$'s are independent uniform r.v. on $[0,2\pi)$. Else, the distributions of the sectors would be Beta distributions $\mathbf{B}(1,m)$ which expectation is $2\pi/(m+1)$. There is thus interaction between the $\theta(i,j)$'s.\\

\par Our conjecture is as follows:
\begin{conj}\label{conj1}Conditionally on $m\in \{2,3,4,5\}$, the vector $(\phi(1),\dots,\phi(m))$ has a distribution close to a symmetric Dirichlet distribution of order $m$ on $[0,2\pi)$ with parameter $\alpha\not= 1$, $\mbox{Dir}(m,[0,2\pi),\alpha)$.
\end{conj}
Symmetric Dirichlet distributions of order $m$ and parameter $\alpha>0$ on $[0,2\pi)$ are probability distributions on $\R^m$ with a support in $\Lambda=\{\eta=(\eta_1,\dots,\eta_m)\in \R^m,\, \sum_{i=1}^m \eta_i=2\pi\}$ and with the following density with respect to the Lebesgue measure on $\Lambda$:
$$f(\eta_1,\dots,\eta_m ; \alpha)=\frac{1}{\mathbf{B}(\alpha)}\prod_{i=1}^{m}\Big(\frac{\eta_i}{2\pi}\Big)^{\alpha-1},\qquad \mbox{ where }\mathbf{B}(\alpha)=\frac{\Big(\int_0^{+\infty}t^{\alpha-1}e^{-t}dt\Big)^m}{\int_0^{+\infty}t^{m\alpha-1}e^{-t}dt}$$
is the Beta function. If we had a Dirichlet distribution conditionally on $m$, the marginal distribution of the exchangeable sectors would be a Beta distribution $\mathbf{B}(\alpha,(m-1)\alpha)$ on $[0,2\pi)$ with expectation $2\pi/m$.
This would also show that the distributions of the asymptotic directions $\theta(i,j)$'s depend only on the number $m$ of unbounded trees and not on the number of offspring of $O$, which is a local phenomenon that is forgotten at large radii.\\

Let us illustrate the conjecture with simulations. We compute the angle between two interfaces and calibrate Beta distributions. Whereas there are no closed form for the maximum-likelihood estimates, the following moment estimates are as follows:
\begin{align}
\widehat{\alpha}= & \frac{\bar{x}}{2\pi}\Big(\frac{\bar{x}(2\pi-\bar{x})}{\mbox{Var}(x)}-1\Big),\qquad \qquad
\widehat{\beta}=  \frac{2\pi - \bar{x}}{2\pi}\Big(\frac{\bar{x}(2\pi-\bar{x})}{\mbox{Var}(x)}-1\Big).
\end{align}
The different densities and the associated Beta approximations are given in Fig. \ref{fig5}.

\begin{figure}[ht]
\begin{center}
\begin{tabular}{ccc}
(a) & (b)  & (c) \\
\includegraphics[width = 0.3\textwidth,trim=0cm 0cm 0cm 0cm]{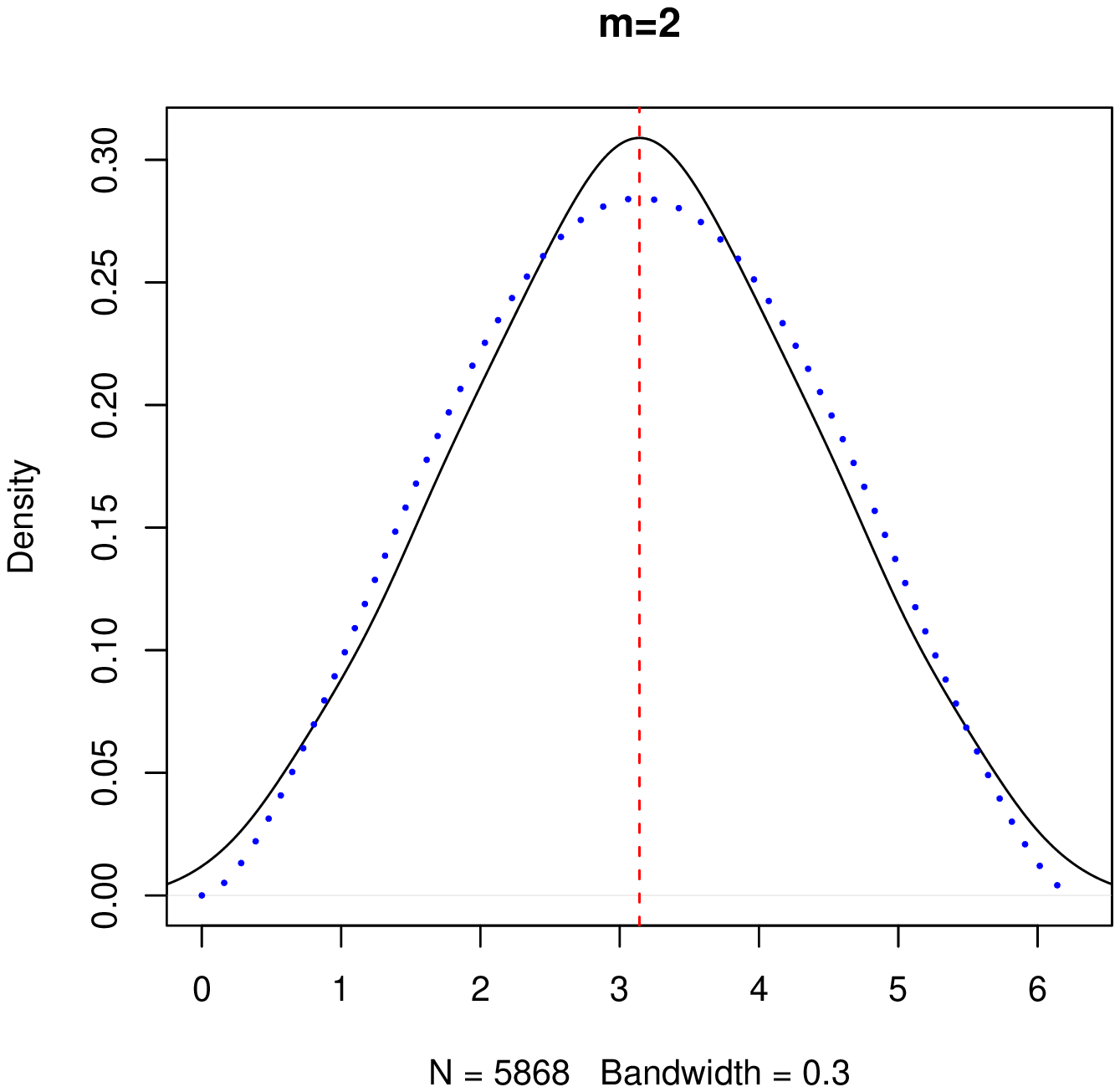} &
\includegraphics[width = 0.3\textwidth,trim=0cm 0cm 0cm 0cm]{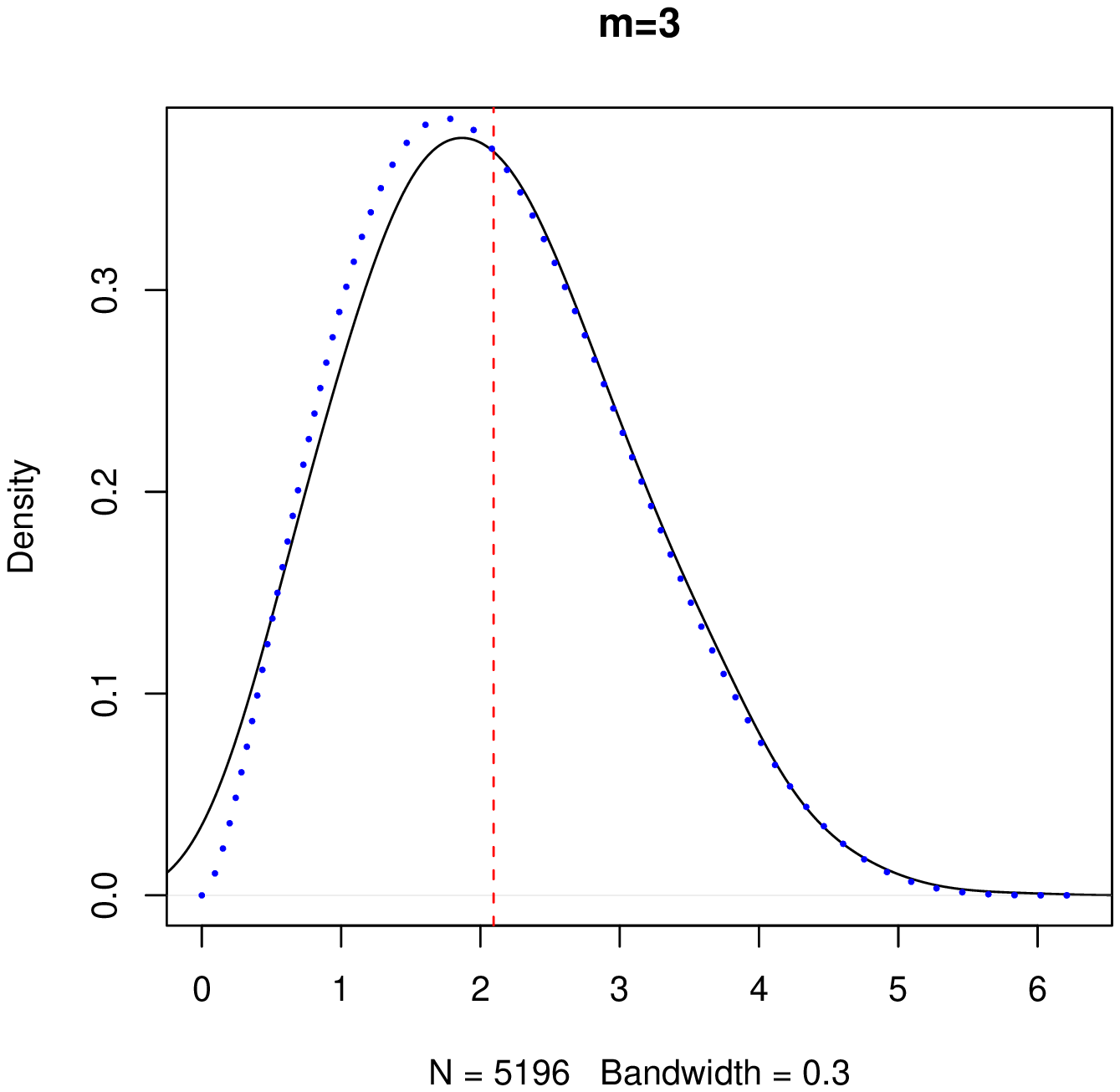} &
\includegraphics[width = 0.3\textwidth,trim=0cm 0cm 0cm 0cm]{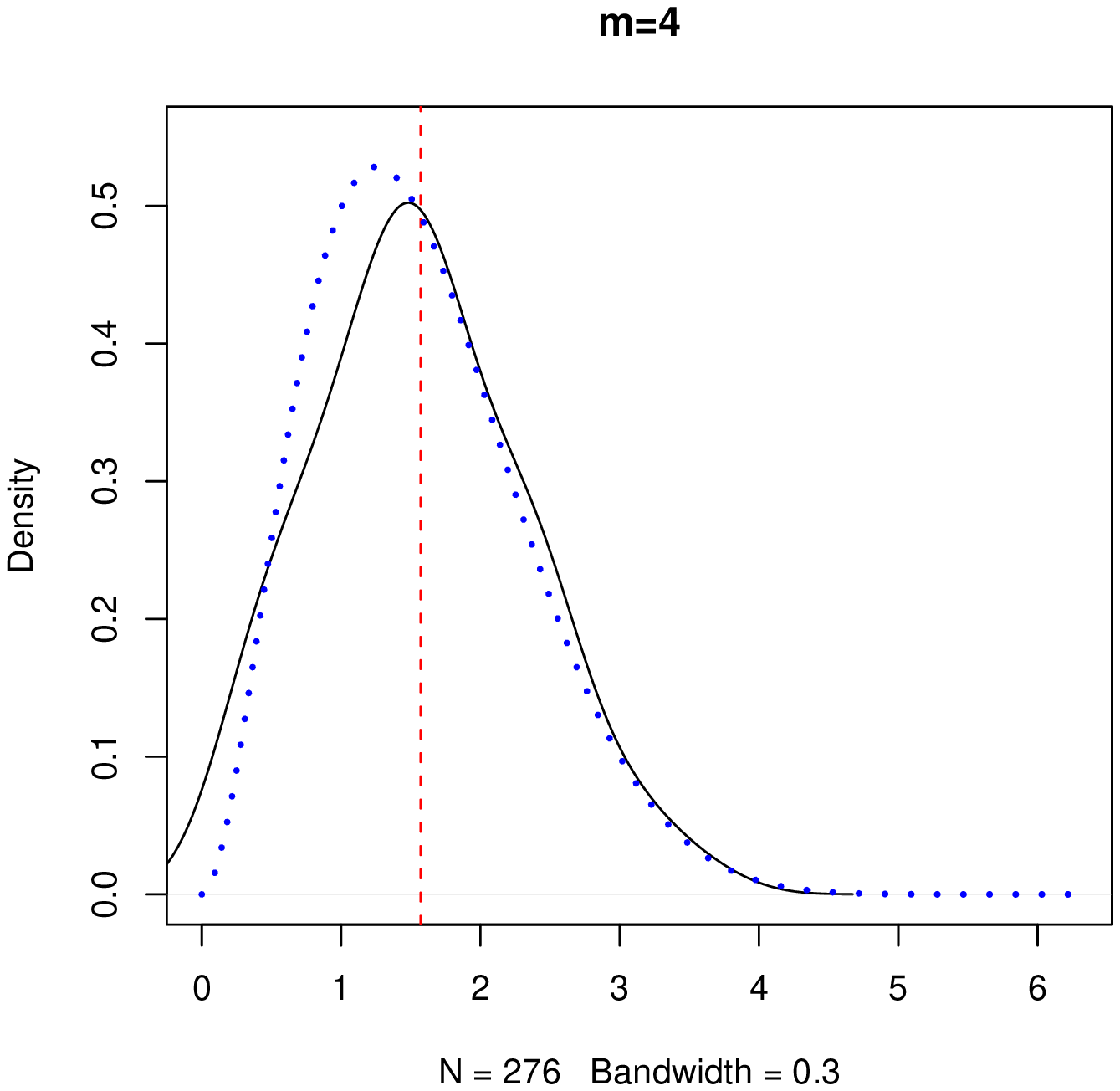} \\
\end{tabular}\end{center}\vspace{-0.5cm}
\caption{{\small \textit{Density estimation for the angles between two interfaces, given the number $m$ of unbounded trees. Given $m$, the distribution concentrates around $2\pi/m$ (red dashed line) and has a smaller variance for greater $m$'s. Calibration with Beta distributions have been carried (blue dotted thick line), but KS test (with test statistic $D$) rejects the null hypothesis of Beta distribution $B(\widehat{\alpha},(m-1)\widehat{\beta})$ in the three cases with p-values smaller than $2.2e-16$ although the distributions look similar graphically. (a) $\widehat{\alpha}=\widehat{\beta}=2.74$, $D = 96\%$; (b) $\widehat{\alpha}=2.69$, $\widehat{\beta}=5.38$, $D = 91.77\%$; (c) $\widehat{\alpha}=2.99$, $\widehat{\beta}=8.99$, $D = 88.79\%$.
}}}
\label{fig5}
\end{figure}

We discuss the case $m=2$ and $m=3$ for which a sufficiently large number of simulations are done to perform statistical tests.\\

\noindent \fbox{$\mathbf{m=2}$} In this case, the joint law of $(\theta(1,2),\theta(2,1))$ is completely described by the distribution of one of the two sectors, say $\phi(1)$. Conditionally on the first interface $\theta(1,2)$, we can wonder whether the other interface is uniformly and independently distributed, \ie whether $\phi(1)$ is a uniform r.v. on $[0,2\pi]$. Testing $H_0$ : $\alpha=\beta=1$ with a likelihood-ratio test provides a test statistic of 2274.93 which leads us to reject the null assumption and hence the independence between the asymptotic direction of the two interfaces. We can easily been convinced of this by looking at Fig. \ref{fig5} (a). As a consequence, the asymptotic directions $\theta(1,2)$ and $\theta(2,1)$ are not independent.\\

\noindent \fbox{$\mathbf{m=3}$} In this case, we performed a $\chi^2$-test for testing the adequation of the joint distribution of the sectors to a Dirichlet distribution. Since the sum of the sectors is equal to $2\pi$, we consider the couple $(\phi(1),\phi(2))$. With our simulations, the $\chi^2$-test statistic is equal to 176.49 and the adequation with the Dirichlet distribution is rejected. However, we can see that as conjectured, the simulated sample looks like a simulated sample from a Dirichlet distribution.

\appendix
\section{Appendix: non-crossing property for the paths of the RST}

\begin{lem}
\label{lemm:croisement}
Any two paths $\gamma$ and $\gamma'$ of the RST (finite or not) cannot cross:
$$
\forall X \in \gamma , \; \forall X' \in \gamma' , \; (X,\A(X)) \cap (X',\A(X')) = \emptyset
$$
(where $(a,b)$ denotes the segment $[a,b]$ in $\mathbb{R}^{2}$ without its endpoints).
\end{lem}

\begin{proof}
Let us assume there exists a point $I$ belonging to both $(X,\A(X))$ and $(Y,\A(Y))$. It is easy to check that this assumption and the construction rule of the RST force $X, Y, \A(X)$ and $\A(Y)$ to be four different points. The same is true for their Euclidean norms with probability one. Moreover, without loss of generality, we can also assume that $|Y|<|X|$. Then, two cases can be distinguished.

\noindent \textit{First case:} If $|\A(X)|<|Y|$ then $Y$ is closer to $\A(Y)$ than $\A(X)$: $|\A(Y) - Y| < |\A(X) - Y|$.
In the same way, the inequality $|\A(Y)|<|Y|<|X|$ implies $
|\A(X) - X| < |\A(Y) - X|$. Now, the triangular inequality leads to a contradiction:
\begin{eqnarray*}
|\A(Y) - Y| + |\A(X) - X| & < & |\A(X) - Y| + |\A(Y) - X| \\
& < & |\A(X) - I| + |I - Y| + |\A(Y) - I| + |I - X| \\
& < & |\A(Y) - Y| + |\A(X) - X| ~.
\end{eqnarray*}

\noindent \textit{Second case:} We now assume that $|Y|<|\A(X)|$ and refer to Fig \ref{fig:croisement2}.
The points $X$ and $\A(X)$ do not belong to the open ball $B(O,|Y|)$ which contains $\A(Y)$ by definition. Hence the existence of the point $I$ forces the segment $(X,\A(X))$ to intersect $S(O,|Y|)$ at two distinct points, say $T_1$ and $T_2$, dividing the closed ball $\overline{B}(O,|Y|)$ in two non overlapping sets, say $U$ and $V$. By hypothesis, each of these two sets contains (exactly) one of the two points $Y$ and $\A(Y)$. Since $|T_1-X|$ and $|T_2-X|$ are smaller than $|X-\A(X)|$ by construction, one of the regions $U$ or $V$ is included in the ball $B(X,|X-\A(X)|)$. So, one of the two points $Y$ and $\A(Y)$ belongs to the ball $B(X,|X-\A(X)|)$. This contradicts the fact that $\A(X)$ is the ancestor of $X$.
\begin{figure}[!ht]
\begin{center}
\psfrag{x}{\small{$X$}}
\psfrag{y}{\small{$Y$}}
\psfrag{o}{\small{$O$}}
\psfrag{I}{\small{$I$}}
\psfrag{Ay}{\small{$\A(Y)$}}
\psfrag{S}{\small{$S(O,|Y|)$}}
\psfrag{SS}{\small{$S(X,|X\!-\!\A(X)|)$}}
\psfrag{t1}{\small{$T_{1}$}}
\psfrag{t2}{\small{$T_{2}$}}
\psfrag{Ax}{\small{$\A(X)$}}
\includegraphics[width=8cm,height=5.5cm]{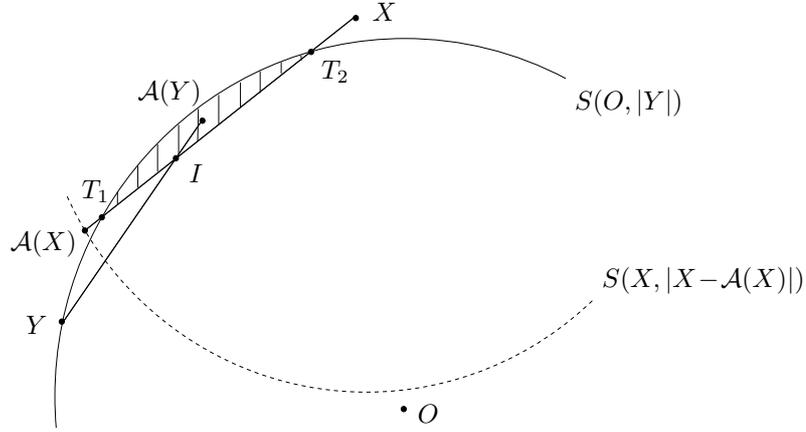}
\end{center}
\caption{\label{fig:croisement2} {\small \textit{The hatched area corresponds to the one of the two sets $U$ and $V$ which is included in the ball $B(X,|X-\A(X)|)$. Here, it contains $\A(Y)$. Besides, let us remark the origin $O$ cannot belong to the ball $B(X,|X-\A(X)|)$.}}}
\end{figure}
\hfill $\Box$ \end{proof}


\noindent \textbf{Acknowledgments:} This work has been financed by the GdR 3477 \textit{G\'eom\'etrie Stochastique}. The authors thank the members of the "Groupe de travail G\'{e}om\'{e}trie Stochastique" of Universit\'{e} Lille 1 and J.-B. Gou\'er\'e for enriching discussions.

%
%
%
%



\providecommand{\noopsort}[1]{}

\end{document}